\documentclass[10pt]{amsart}

\usepackage[english]{babel}
\usepackage{latexsym,amsfonts}
\usepackage{amssymb}
\usepackage{amsmath}
\usepackage{amsthm}

\newtheorem{lemma}{Lemma}[section]
\newtheorem{theorem}[lemma]{Theorem}
\newtheorem{proposition}[lemma]{Proposition}
\newtheorem{cor}[lemma]{Corollary}

\theoremstyle{remark}
\newtheorem{rem}[lemma]{Remark}

\theoremstyle{definition}

\newtheorem{definition}[lemma]{Definition}
\numberwithin{equation}{section}

\def\Mat{\mathop{\rm Mat}\nolimits}          

\def\GL{\mathop{\rm GL}\nolimits}  
\def\AGL{\mathop{\rm AGL}\nolimits}           
  
\def\char{{\rm char\,}}
\def\Ker{{\rm Ker\,}}

\def\diag{\mathop{\rm diag}\nolimits}
\def\rk{\mathrm{rk}}

\def\sp{\mathrm{Span}}

\def\s{\sharp}

\def\T{\mathop{\rm T}\nolimits}

\newcommand{\R}{\mathbb{R}}   
\newcommand{\RR}{\mathcal{R}}  
  
\newcommand{\F}{\mathbb{F}}    
\newcommand{\A}{\mathcal{L}}    
 \newcommand{\C}{\mathbf{C}}    
 \newcommand{\J}{\mathbf{J}}    
\newcommand{\Tr}{\mathcal{T}}    
\newcommand{\B}{\mathcal{B}}   
 \newcommand{\Z}{\mathbf{Z}}

\def\d{\mathsf{d}}
\def\r{\mathsf{r}}
\def\k{\mathsf{k}}
\def\G{\overline{\mathsf{B}}}
\def\H{\mathsf{B}}
\def\M{\mathsf{M}}
\def\N{\mathsf{N}}
\def\S{\mathsf{S}}
\def\P{\mathsf{P}}
\def\QQ{\mathsf{Q}}
\def\K{\mathsf{C}}
\def\CC{\mathsf{A}}
\def\CCC{\overline{\mathsf{A}}}

\def\D{\Delta_4(\F)}
\def\rel{\divideontimes}
\begin{document}

\title[Isomorphism classes of four dimensional nilp. assoc. algebras]{Isomorphism classes of four dimensional nilpotent associative algebras over a field}

\author{Marco Antonio Pellegrini}
\address{Dipartimento di Matematica e Fisica, Universit\`a Cattolica del Sacro Cuore, Via Musei 41,
I-25121 Brescia, Italy}
\email{marcoantonio.pellegrini@unicatt.it}

 
\begin{abstract}
In this paper we classify the isomorphism classes of four dimensional nilpotent  associative algebras over a field $\F$,
studying regular subgroups of the affine group $\AGL_4(\F)$.
In particular  we provide explicit representatives for such classes  when $\F$ is  a finite 
field, the real field $\R$ or an algebraically closed field.
\end{abstract}

\maketitle

\section{Introduction}\label{intro}

The aim of this paper is the classification of the isomorphism classes of four dimensional nilpotent associative  
algebras over a field $\F$. This result will be achieved exploiting the properties of the regular subgroups of 
the affine group $\AGL_4(\F)$. 

It is worth recalling that the affine group $\AGL_n(\F)$ can be identified with the 
subgroup of $\GL_{n+1}(\F)$ consisting of the invertible
matrices 
having  $(1,0,\ldots,0)^{\T}$ as first column. It follows that $\AGL_n(\F)$ acts on the right on the set 
$\mathcal{A}=\left\{(1,v): v\in \F^n\right\}$ of affine points.
A subgroup $R$ of $\AGL_n(\F)$ is called regular if it 
acts regularly on $\mathcal{A}$, namely  if for every $v\in \F^n$ there exists a unique
element in  $R$ having $(1,v)$ as first row. 
Writing the elements $r$ of a regular subgroup $R\leq \AGL_n(\F)$  
as
\begin{equation}\label{tauv}
r=\begin{pmatrix} 1 & v \\ 0 & I_n+\delta_R(v)  \end{pmatrix}=\mu_R(v), 
\end{equation}
we give rise to the functions $\delta_R: \F^n\rightarrow \Mat_n(\F)$ and $\mu_R: \F^n\rightarrow R$.
For instance, the translation subgroup $\Tr_n$ of $\AGL_n(\F)$ is a regular subgroup, corresponding to the choice 
$\delta_{\Tr_n}$ equal to the zero function. 

We focus our attention on the set  $\Delta_n(\F)$  of the regular subgroups $R$ of 
$\AGL_n(\F$) with the property that $\delta_R$ is a linear function. Notice that if $R$ is abelian, then $R\in \Delta_n(\F)$ (see 
\cite{CDS,T}); however, for $n\geq 3$, the set $\Delta_n(\F)$ also contains nonabelian groups (see \cite{PT}).
We also remark that if $R\in \Delta_n(\F)$, then $R$ is unipotent and contains nontrivial translations (see 
\cite{P,PT2}).
Our main motivation for studying the set $\Delta_n(\F)$ is that  there exists a bijection 
between  the set of isomorphism classes of nilpotent  associative algebras of dimension $n$ over a field $\F$ and the set of conjugacy classes 
of  subgroups in $\Delta_n(\F)$ (see \cite[Proposition 2.2]{P}). 
In fact, our classification of four dimensional nilpotent associative $\F$-algebras (Theorem \ref{main}) will follow from the classification of the conjugacy 
classes of regular subgroups in $\D$.

The paper is organized as follows.
In Section \ref{prel} we recall some useful results concerning $\Delta_n(\F)$ and 
in the following sections we classify, up to conjugation, the regular subgroups in $\D$, 
see Theorems \ref{main_ab} and \ref{main_nonab}.
More in detail, in Section \ref{ab} we deal with abelian regular subgroups and in Section \ref{nonab} we consider nonabelian subgroups.  
Corollaries \ref{ab_ac} and \ref{ab_f} and Propositions \ref{non_ac} and \ref{non_f} give the explicit classification 
for  algebraically closed fields or finite fields.
The classification for the real field is given in Corollary \ref{ab_real} and Proposition \ref{nonab_real}.

Before to describe the results we obtained,  we need to fix some notation.
The sets $\Pi_S(\F)$ and $\Pi_A(\F)$ are  complete sets of representatives of  projective congruent classes for, 
respectively,
symmetric and asymmetric matrices in $\Mat_3(\F)$, see  Definition \ref{projcong}.
Given a field $\F$,  we denote by $\F^\square$ a transversal for $(\F^\ast)^2$ in $\F^\ast$.
Also,  for any $\rho \in 
\F^\square$,  we denote by $\CC_\rho(\F)$ a fixed set of representatives for the equivalence classes with respect to 
the following relation  $\rel_\rho$ defined on $\CCC(\F)=\left(\F\setminus\{1\}\right)\times \F$:
given $(\beta_1,\beta_2),(\beta_3,\beta_4)\in \CCC(\F)$, we write $(\beta_1,\beta_2)\rel_\rho (\beta_3,\beta_4)$ if and 
only if 
either $(\beta_3,\beta_4)=(\beta_1,\pm \beta_2)$ or
\begin{equation}
\beta_3=\frac{\beta_1 t^2+ \beta_2 t-\rho}{t^2 + \beta_2 t -\rho\beta_1}\quad\textrm{ and }\quad
 \beta_4= \pm \frac{\beta_2 t^2-2\rho(\beta_1+1)t-\rho\beta_2}
  {t^2+\beta_2 t -\rho\beta_1}
\end{equation}
for some $t \in \F$  such that $(t^2+\rho )(t^2 + \beta_2 t -\rho\beta_1)\neq 0$.
If $\char\F=2$, we fix two additional sets $\H(\F)$ and $\K(\F)$. The set $\H(\F)$ is a transversal for the subgroup 
$\G(\F)=\{\lambda^2+\lambda : \lambda \in \F\}$ in $\F$. The set $\K(\F)$ is a set of representatives for 
the equivalence classes with respect to the following relation  $\star$  defined on $\F^\square$: given 
$\rho_1,\rho_2\in 
\F^\square$, we write
$\rho_1 \star \rho_2$ if and only if $\rho_2=\frac{x_1^2+x_2^2\rho_1}{x_3^2+x_4^2\rho_1}$ for some $x_1,x_2,
x_3, x_4 \in \F$ such that $x_1x_4\neq x_2x_3$. For convenience, we assume $1\in \F^\square$ and $0\in \H(\F)$.

We can now state our main result.

\begin{theorem}\label{main}
The distinct isomorphism classes of  four dimensional  nilpotent associative  algebras over a field $\F$ can 
be represented by the algebras described in the second column of Table \ref{tab1} (abelian case) and 
Table \ref{tab2} (nonabelian case).
\end{theorem}

\begin{table}[t]
\begin{footnotesize}
$$\begin{array}{ccc}
R & R-I_5  & \textrm{Conditions} \\ \hline
S_{(5)} & 
\begin{pmatrix}
0 & x_1 & x_2 & x_3 & x_4 \\
0 &   0 & x_1 & x_2 & x_3\\
0 &   0 &   0 & x_1 & x_2\\
0 &   0 &   0 &   0 & x_1 \\
0 &   0 &   0 &   0 &   0
 \end{pmatrix}  
& \\[-10pt] \\ 
 S_{(4,1)} & \begin{pmatrix}
0 & x_1 & x_2 & x_3 & x_4 \\
0 &   0 & x_1 & x_2 & 0\\
0 &   0 &   0 & x_1 & 0\\
0 &   0 &   0 &   0 & 0 \\
0 &   0 &   0 &   0 &   0
\end{pmatrix} 
 & \\[-10pt] \\
S_{(3,2)}^{\s} & \begin{pmatrix}
0 & x_1 & x_2 & x_3 & x_4 \\
0 &   0& x_1 & 0 & x_2\\
0 &   0 &   0 & 0 & x_1\\
0 &   0 &   0 &   0 & x_3 \\
0 &   0 &   0 &   0 &   0
\end{pmatrix}\\[-10pt] \\
U_1(0,0) & \begin{pmatrix}
0 & x_1 & x_2 & x_3 & x_4\\
0 & 0 & 0 & x_1 & x_2\\
0 & 0 & 0 & 0 & x_1\\
0 & 0 & 0 & 0 & 0\\
0 & 0 & 0 & 0 & 0\\
\end{pmatrix}
&\char \F\neq 2\\[-10pt] \\
U_1(0,\rho) & \begin{pmatrix}
0 & x_1 & x_2 & x_3 & x_4\\
0 & 0 & 0 & x_1 & x_2\\
0 & 0 & 0 & \rho x_2 & x_1\\
0 & 0 & 0 & 0 & 0\\
0 & 0 & 0 & 0 & 0\\
\end{pmatrix}
&
\begin{array}{ll}
\rho \in \F^\square  & \textrm{if } \char \F\neq 2\\
\rho \in \K(\F)  &\textrm{if } \char \F= 2\\
\end{array}\\[-10pt] \\
U_1(1,\epsilon)&
\begin{pmatrix}
0 & x_1 & x_2 & x_3 & x_4\\
0 & 0 & 0 & x_1 & x_2\\
0 & 0 & 0 & \epsilon x_2 & x_1+ x_2\\
0 & 0 & 0 & 0 & 0\\
0 & 0 & 0 & 0 & 0\\
\end{pmatrix}
&  
\char \F = 2, \; \epsilon \in \H(\F) 
\\[-10pt] \\
\RR_{D} & 
\begin{pmatrix}
0 & x_1 & x_2 & x_3 & x_4\\
0 & 0 & 0 & 0 & \alpha_1x_1+\alpha_2x_2+\alpha_3 x_3\\
0 & 0 & 0 & 0 & \alpha_2 x_1+\beta_2 x_2+\beta_3x_3\\
0 & 0 & 0 & 0 & \alpha_3 x_1+\beta_3 x_2+\gamma_3 x_3\\
0 & 0 & 0 & 0 & 0
\end{pmatrix}
&  \begin{pmatrix}
\alpha_1 & \alpha_2 & \alpha_3\\
\alpha_2 & \beta_2 & \beta_3\\
\alpha_3 & \beta_3 & \gamma_3 
\end{pmatrix}\in \Pi_S(\F)\\[-10pt] \\ 
\end{array}$$
\end{footnotesize}
\caption{Four dimensional abelian nilpotent associative  $\F$-algebras and abelian regular subgroups of $\AGL_4(\F)$.}
\label{tab1}
\end{table}

\begin{table}[t]
\begin{footnotesize}
$$\begin{array}{ccc}
R & R-I_5  & \textrm{Conditions} \\ \hline
 U_2(0,1,0,1,1) &  \begin{pmatrix}
 0 &    x_1 &   x_2 &   x_3 &         x_4 \\
 0 &       0 &   x_1 &       0 &   x_2+  x_3 \\ 
 0 &       0 &      0 &        0 &         x_1 \\ 
 0 &       0 &      0 &        0 &         x_3 \\ 
 0 &       0 &      0 &        0 &         0 
  \end{pmatrix}
  & \\[-10pt] \\
U_2(0,1,0,0,1) & \begin{pmatrix}
 0 &   x_1 &  x_2 &  x_3 &    x_4 \\
 0 &      0 &  x_1 &      0 &   x_2+ x_3 \\ 
 0 &      0 &     0 &      0 &    x_1 \\ 
 0 &      0 &     0 &      0 &    0\\ 
 0 &      0 &     0 &      0 &    0 
  \end{pmatrix}
  & \\[-10pt] \\ 
  U_3(0,1,0) & \begin{pmatrix}
0 & x_1 & x_2 & x_3 & x_4\\
0 & 0 & x_1 & 0 & 0\\
0 & 0 & 0 & 0 & 0 \\
0 & 0 & 0 & 0 &  x_1\\
0 & 0 & 0 & 0 & 0
\end{pmatrix}
  & \\[-10pt] \\ 
U_3(0,1,1) & \begin{pmatrix}
0 & x_1 & x_2 & x_3 & x_4\\
0 & 0 & x_1 & 0 & 0\\
0 & 0 & 0 & 0 & 0 \\
0 & 0 & 0 & 0 &  x_1 + x_3\\
0 & 0 & 0 & 0 & 0
\end{pmatrix}
  & \\[-10pt] \\ 
U_3(1,\lambda,0) & \begin{pmatrix}
0 & x_1 & x_2 & x_3 & x_4\\
0 & 0 & x_1 & 0 &  x_3\\
0 & 0 & 0 & 0 & 0 \\
0 & 0 & 0 & 0 & \lambda x_1 \\
0 & 0 & 0 & 0 & 0
\end{pmatrix}
  & \lambda \in \F\setminus\{1\} \\[-10pt] \\ 
  U_4(\rho,\beta_1,\beta_2) & 
  \begin{pmatrix}
1 & x_1 & x_2 & x_3 & x_4\\
0 & 1 & 0 & x_1 &  x_2\\
0 & 0 & 1 & \rho x_2 & \beta_1 x_1+\beta_2 x_2\\
0 & 0 & 0 & 1 & 0\\
0 & 0 & 0 & 0 & 1
\end{pmatrix}
  & \rho \in \F^\square,\;\;(\beta_1,\beta_2)\in \CC_\rho(\F)  \\[-10pt] \\
  U_5(1,1,\epsilon) & \begin{pmatrix}
0 & x_1 & x_2 & x_3 & x_4\\
0 & 0 & 0 & x_2 &  x_1+ x_2\\
0 & 0 & 0 & x_1 & \epsilon x_2\\
0 & 0 & 0 & 0 & 0\\
0 & 0 & 0 & 0 & 0
\end{pmatrix}
  & \char\F=2,\; \;\epsilon \in \H(\F) \\[-10pt] \\ 
 
\RR_{D} & 
\begin{pmatrix}
0 & x_1 & x_2 & x_3 & x_4\\
0 & 0 & 0 & 0 & \alpha_1x_1+\alpha_2x_2+\alpha_3 x_3\\
0 & 0 & 0 & 0 & \beta_1 x_1+\beta_2 x_2+\beta_3x_3\\
0 & 0 & 0 & 0 & \gamma_1 x_1+\gamma_2 x_2+\gamma_3 x_3\\
0 & 0 & 0 & 0 & 0
\end{pmatrix}
&  \begin{pmatrix}
\alpha_1 & \alpha_2 & \alpha_3\\
\beta_1 & \beta_2 & \beta_3\\
\gamma_1 & \gamma_2 & \gamma_3 
\end{pmatrix}\in \Pi_A(\F)\\[-10pt] \\ 
\end{array}$$
\end{footnotesize}
\caption{Four dimensional nonabelian nilpotent associative  $\F$-algebras and nonabelian regular subgroups in $\D$.}
\label{tab2}
\end{table}

In particular, there are exactly 
$$\left\{\begin{array}{cl}
9 & \textrm{ if } \F= \overline{\F} \textrm{ is algebraically closed and } \char \F \neq 2,\\
10 & \textrm{ if } \F= \overline{\F} \textrm{ is algebraically closed and } \char \F = 2,\\
11 & \textrm{ if } \F= \F_q \textrm{ is finite},\\
12 & \textrm{ if } \F= \R,\\
\end{array}\right.$$
nonisomorphic abelian nilpotent associative  $\F$-algebras of dimension four.
We point out that the classification given in Theorem \ref{main} for the abelian case was already obtained, using other techniques, by De Graaf 
\cite{DG} (if $\F=\F_q$ or $\R$) and by Poonen \cite{Poo} (for algebraically closed fields).
It is also interesting to observe that if $\F=\F_q$ is a finite field, then there are exactly $5q+9$ (if $q$ is odd) 
or $5q+6$ (if $q$ is even) nonisomorphic nonabelian nilpotent associative  $\F_q$-algebras of dimension four (see 
Proposition \ref{non_f}).

\section{Preliminary results}\label{prel}

Our classification of regular subgroups in $\D$ basically relies on the key results of \cite{P,PT} that we recall 
in this section.

\begin{proposition}\cite[Theorems 2.4 and 4.4]{PT}\label{lem1}
Let $R\in \Delta_n(\F)$. Then,
\begin{itemize}
\item[(a)] $R$ is a unipotent subgroup;
\item[(b)] up to conjugation under $\AGL_n(\F)$, the center $\Z(R)$ of $R$ contains a nontrivial element $z$ that coincides with its Jordan form $\diag\left(J_{n_1}, \dots, J_{n_k}\right)$ having upper unitriangular Jordan blocks  $J_{n_i}$ of respective sizes  $n_i\ge n_{i+1}$ for all $i\ge 1$.
\end{itemize}
\end{proposition}

If  $\delta_R$ is a linear function then a regular subset $R$ of $\AGL_n(\F)$  is a  subgroup if and only if
\begin{equation*}
\delta_R(v\delta_R(w))=\delta_R(v)\delta_R(w),\ \quad  \textrm{ for all }  v,w\in \F^n,
\end{equation*}
see \cite{PT}. Also, given $v,w\in \F^n$, 
\begin{equation*}
\mu_R(v)\mu_R(w)=\mu_R(w)\mu_R(v) \quad\textrm{ if and only if }\quad v\delta_R(w)= w\delta_R(v).
 \end{equation*}

We also recall that two subgroups $R_1,R_2\in\Delta_n(\F)$ are conjugate in $\AGL_n(\F)$ if and only if there 
exists a matrix $g=\diag(1,\bar g)$ with $\bar g\in \GL_n(\F)$ such that $R_1^g=R_2$ (see \cite[Proposition 3.4]{PT}).
On the other hand it is useful, mainly to prove that two regular subgroups are not conjugate, to
introduce the following three parameters. Let $H\in \{R, \Z(R)\}$, where  $R\in \Delta_n(\F)$. Then we may define:
\begin{equation*}
\begin{array}{rcl}
\d(H) & = &\max\{\deg \min_{\F}(h-I_{n+1}):  h \in H\}; \\
\r(H) & = &\max\{\rk(h-I_{n+1}) : h \in H \}; \\
\k(H) & = &\dim \Ker(\delta_R|_H).
\end{array}
\end{equation*}
The last parameter is justified since the set $\{v\in \F^n: \mu_R(v)\in \Z(R)\}$ is a subspace of $\F^n$ (see 
\cite[Proposition 2.1]{P}). 

We will proceed considering the possible values of $(\d(\Z(R)), \r(\Z(R))$ for $R\in \Delta_4(\F)$ and working in the 
centralizer of a Jordan form as described in Proposition \ref{lem1}(b). Unfortunately,  the case  $\r(\Z(R))=2$ will require a 
different approach, based on the classification of the regular subgroups in $\Delta_3(\F)$ obtained in \cite{PT}.
In particular, we will make use of the following observation.

Any regular subgroup $\widetilde{R}\in \Delta_n(\F)$, $n\geq 2$, can be written as 
\begin{equation}\label{Rtilde}
\widetilde{R}=\RR(R,D)=\left\{ 
\begin{pmatrix} 1 & X & x_{n} \\ 0 & I_{n-1}+\delta_R(X) & DX^{\T} \\ 0 & 0 & 1\end{pmatrix}: X \in \F^{n-1}, \; x_n \in 
\F \right\},
\end{equation}
for some  $R\in  \Delta_{n-1}(\F)$  and some square matrix $D\in \Mat_{n-1}(\F)$ such that
\begin{equation}\label{dt}
D\delta_R(e_i)^{\T} e_j^{\T}=\delta_R(e_j)De_i^{\T}\qquad \textrm{for all } i,j=1,\ldots,n-1
\end{equation}
($\{e_1,\ldots,e_{n-1}\}$ denotes the canonical basis of $\F^{n-1}$), see \cite{P}. In particular, we can take 
$R=\Tr_{n-1}$: in this case any square matrix $D \in \Mat_{n-1}$ satisfies \eqref{dt}. We set
\begin{equation}\label{eqRD}
 \RR_D:=\RR(\Tr_{n-1},D)=\left\{ 
\begin{pmatrix} 1 & X & x_{n} \\ 0 & I_{n-1} & DX^{\T} \\ 0 & 0 & 1\end{pmatrix}: X \in \F^{n-1}, \; x_n \in \F 
\right\}.
\end{equation}

The study of these subgroups $\RR_D$ is justified mainly because of their connection with projectively congruent matrices.

\begin{definition}\cite{W1}\label{projcong}
Two matrices $A,B\in  \Mat_n(\F)$ are said to be \emph{projectively congruent} if $PAP^{\T}=\lambda B$, for some non-zero element $\lambda\in \F^\ast$ and some invertible matrix $P\in \GL_n(\F)$.
\end{definition}

\begin{lemma}\cite[Lemma 3.2]{P}\label{RD}
Given two matrices $A,B\in \Mat_{n-1}(\F)$, the subgroups $\RR_A$ and $\RR_B$ are 
conjugate in $\AGL_n(\F)$ if and only if $A$ and $B$ are projectively congruent. 
\end{lemma}

Furthermore, it is easy to see that
\begin{eqnarray*}
\d(\RR_D)& =& \left\{ \begin{array}{cl} 2& \textrm{ if 
 } D^{\T}=-D \textrm{ and }D \textrm{ has zero diagonal}, \\ 3 & \textrm{ otherwise},
\end{array}\right.\\
\r(\RR_D) & =& \left\{ \begin{array}{cl} 2& \textrm{ if 
 } D\neq 0, \\ 1 & \textrm{ if } D=0,
\end{array}\right.\\
\k(\RR_D) & =& n-\rk(D).
\end{eqnarray*}

A given  nilpotent associative algebra $N$ of dimension $n$ can be embedded into $\Mat_{n+1}(\F)$ via 
$$m \mapsto \begin{pmatrix} 0 & m_{\B}\\ 0 & \delta(m)  \end{pmatrix},$$
where, for any $m\in N$, $m_{\B}$ and $\delta(m)$ denote, respectively, the coordinate row vector of $m$ and the matrix
of the right multiplication by $m$ with respect to a fixed basis $\B$ of $N$ over $\F$.
Identifying $N$ with its image, $\A=\F I_{n+1}+N$ is a split local subalgebra of $\Mat_{n+1}(\F)$ with Jacobson radical $\J(\A)=N$. The subset 
$R=\{I_{n+1}+m : m \in N\}\subseteq \A$ consists  of invertible matrices and is closed under multiplication,
since $\delta(m_1 \delta(m_2))=\delta(m_1)\delta(m_2)$ for all $m_1,m_2\in N$. Hence, $R$ is a regular subgroup lying in $\Delta_n(\F)$, by \cite[Lemma 2.1]{PT}. On the other hand, given a regular subgroup $R\in\Delta_n(\F)$, we have that 
\begin{equation}\label{LR}
\A_R=\F I_{n+1}+R
\end{equation}
is a split local $\F$-algebra of dimension $n$, by \cite[Theorem 3.3]{PT}. Notice that 
set
$$N=\J(\A_R)=R-I_{n+1}= \left\{\begin{pmatrix} 0& v\\ 0 & \delta_R(v)   \end{pmatrix}:   v\in 
\F^n\right\}$$
is a nilpotent  associative $\F$-algebra of dimension $n$.

\begin{proposition}\cite{P,PT}\label{prop34}
Let  $R_1,R_2\in \Delta_n(\F)$ and $\A_{R_1},\A_{R_2}$ be as in \eqref{LR}. 
Then the following conditions are equivalent:
\begin{itemize}
\item[(\rm{a})] the subgroups $R_1$ and $R_2$ are conjugate in $\AGL_n(\F)$;
\item[(\rm{b})] the split local algebras $\A_{R_1}$ and $\A_{R_2}$  are isomorphic;
\item[(\rm{c})] the nilpotent associative algebras $\J(\A_{R_1})$ and $\J(\A_{R_2})$  are isomorphic.
\end{itemize}
\end{proposition}

For sake of brevity, we write
\begin{equation*}\label{Ubreve}
R=\begin{pmatrix}
1&\begin{array}{cccc}
x_1 & x_2 & \ldots & x_n   
  \end{array}
\\
0&I_n+\delta_R(x_1 , x_2,  \ldots,  x_n )
\end{pmatrix}
\end{equation*}
to indicate the regular subgroup
$$R=\left\{\begin{pmatrix}
1&\begin{array}{cccc}
x_1 & x_2 & \ldots & x_n   
  \end{array}
\\
0&I_n+\delta_R(x_1 , x_2,  \ldots,  x_n )
\end{pmatrix}:  x_1,\ldots,x_n\in \F\right\}.$$
For all $i\le n$, we denote by $X_i$ the matrix of $R-I_{n+1}$ obtained taking $x_i=1$ and $x_j=0$ for all $j\neq i$.
Finally, $E_{i,j}$ is the elementary matrix having $1$ at position $(i,j)$, $0$ elsewhere.

\section{The regular subgroups in $\D$}\label{sec2}

For classifying the regular subgroups  $R\in \D$, where $\F$ is any field, 
we consider the possible values of $(\d(\Z(R)),\r(\Z(R)))$. 
By Proposition \ref{lem1} 
the subgroup $R$ is unipotent and its center contains an element conjugate in $\AGL_4(\F)$ 
to  one of the following Jordan forms $z$:
\begin{center}
\begin{tabular}{ccc}
$z$ & $\d(\Z(R))$ & $\r(\Z(R))$\\\hline 
$J_5$ & $5$ & $4$\\
$\diag(J_4,I_1)$ & $4$ & $3$\\
$\diag(J_3,J_2)$ & $3$ & $3$  \\
$\diag(J_3,I_2)$ & $3$ & $2$\\
$\diag(J_2,J_2,I_1)$ & $2$ & $2$\\
$\diag(J_2,I_3)$ & $2$ & $1$
\end{tabular}
\end{center}
Some of these cases have been already studied in \cite{PT}: we recall here the results obtained in 
Lemmas 5.2, 5.4  and 7.4. First of all, if $\d(\Z(R))\geq 4$, then $R$ is abelian. In particular, if 
$\d(R)=5$, then $R$ is conjugate to 
$$S_{(5)}=\begin{pmatrix}
1 & x_1 & x_2 & x_3 & x_4 \\
0 &   1 & x_1 & x_2 & x_3\\
0 &   0 &   1 & x_1 & x_2\\
0 &   0 &   0 &   1 & x_1 \\
0 &   0 &   0 &   0 &   1
          \end{pmatrix}.$$
If   $\d(R)=4$, we have two conjugacy classes of regular subgroups:
when $\k(R)=2$ the subgroup $R$ is conjugate to $S_{(4,1)}$ and
when $\k(R)=1$ it is conjugate to $S_{(3,2)}^{\s}$, where
$$S_{(4,1)}=\begin{pmatrix}
1 & x_1 & x_2 & x_3 & x_4 \\
0 &   1 & x_1 & x_2 & 0\\
0 &   0 &   1 & x_1 & 0\\
0 &   0 &   0 &   1 & 0 \\
0 &   0 &   0 &   0 &   1
\end{pmatrix} \quad \textrm{ and } \quad S_{(3,2)}^{\s}=\begin{pmatrix}
1 & x_1 & x_2 & x_3 & x_4 \\
0 &   1 & x_1 & 0 & x_2\\
0 &   0 &   1 & 0 & x_1\\
0 &   0 &   0 &   1 & x_3 \\
0 &   0 &   0 &   0 &   1
\end{pmatrix}.$$
Now, if $\d(\Z(R))=\r(\Z(R))=3$, then again $R$ is abelian.  Conjugating the subgroups obtained in \cite[Lemma 7.4]{PT} by 
$g=\diag(I_2,E_{1,2}+E_{2,1},1)$, we obtain that  $R$ is conjugate to
\begin{equation}\label{U1}
U_1(\alpha,\beta)=\begin{pmatrix}
1 & x_1 & x_2 & x_3 & x_4\\
0 & 1 & 0 & x_1 & x_2\\
0 & 0 & 1 & \beta x_2 & x_1+\alpha x_2\\
0 & 0 & 0 & 1 & 0\\
0 & 0 & 0 & 0 & 1\\
\end{pmatrix}
\end{equation}
for some $\alpha,\beta\in \F$.

If $\d(\Z(R))=3$ and $\r(\Z(R))=2$, we may assume that $z=\diag(J_3 , 
I_2)^g\in \Z(R)$, where $g=\diag(I_2,E_{1,3}+E_{2,1}+E_{3,2})$. So, working  in $\C_{\AGL_4(\F)}(z)$ and 
using the linearity of 
$\delta_R$ we obtain that
$R$ is 
conjugate to a subgroup $\RR_D$, where 
$$D=\begin{pmatrix} 1 & 0 & 0 \\ 0 & \beta_2 & \beta_3 \\ 0 & \gamma_2& \gamma_3
    \end{pmatrix},$$
for some $\beta_2,\beta_3,\gamma_2,\gamma_3\in \F$.

If $\d(\Z(R))=\r(\Z(R))=2$, we may assume that the center of $R$ contains the 
element $z=\diag(J_2,J_2,I_1)^g$, where $g=\diag(1,E_{1,2}+E_{2,1},E_{1,2}+E_{2,1})$.
Again, working in $\C_{\AGL_4(\F)}(z)$ and using the linearity of $\delta_R$ we obtain that $R$ is conjugate to the 
subgroup
\begin{equation}\label{U2}
U_2(\alpha_1,\alpha_3,\gamma_1,\gamma_3,\zeta)=\begin{pmatrix}
 1 &                 x_1 &                 x_2 &                 x_3 &                 x_4 \\
 0 &                  1 &             \zeta x_1 &                  0 &   \alpha_1 x_1 +x_2+ \alpha_3 x_3 \\ 
 0 &                  0 &                  1 &                  0 &                 x_1 \\ 
 0 &                  0 &                  0 &                  1 &      \gamma_1 x_1 +\gamma_3 x_3 \\ 
 0 &                  0 &                  0 &                  0 &                  1 
  \end{pmatrix}
\end{equation}
for some $\alpha_1,\alpha_3,\gamma_1,\gamma_3,\zeta\in \F$.

Finally, if $R$ is abelian and $(\d(R),\r(R))=(2,1)$, then  $R=\Tr_4$ by \cite[Lemma 5.3]{PT}.

\section{The abelian case}\label{ab}

In this section we assume that $R\in \D$ is abelian.
In view of the results recalled in Section \ref{sec2}, we are left to determine the conjugacy classes of subgroups of 
type $U_1(\alpha,\beta)$, $U_2(\alpha_1,\alpha_3,\gamma_1,\gamma_3,\zeta)$ and $\RR_D$ where $D\in \Mat_3(\F)$ is a 
symmetric matrix.
We start with the subgroups $U_1(\alpha,\beta)$.

\begin{lemma}\label{J3J2}
Let $R$ be an abelian regular subgroup of $\AGL_4(\F)$ such that $\d(R)=\r(R)=3$. 
Then  $R$ is conjugate to exactly one of the following subgroups:
$$U_1(0,\rho) \; \textrm{ with } \rho \in \F^\square\cup \{0\}$$
if $\char\F\neq 2$;
$$U_1(1,\epsilon) \; \textrm{ with  } \epsilon \in \H(\F)\quad  \textrm{ or } \quad U_1(0,\rho) \; \textrm{ with } \rho \in \K(\F)$$
if $\char \F=2$.
\end{lemma}

\begin{proof}
By the previous argument, we can assume that $R=U_1(\alpha,\beta)$, as in \eqref{U1}, for some $\alpha,\beta \in \F$. 
First, suppose that $\char \F\neq 2$ and set $\Delta=\alpha^2+4\beta$. If $\Delta\neq 0$, write 
$\Delta=\omega^2\rho$ for some $\omega \in \F^\ast$ and a unique $\rho \in \F^\square$.
We can define an 
epimorphism (of algebras)
$\Psi: \F[t_1, t_2]\rightarrow \A_{U_1(\alpha,\beta)}$ by setting
\begin{equation*}
\Psi(t_1) =  \omega \rho X_1,\qquad
\Psi(t_2)  =  \alpha X_1- 2X_2.
\end{equation*}
We have $\Ker(\Psi)=\langle t_1^3,\; t_2^3,\; t_1^2-\rho t_2^2\rangle$.
The subgroup $U_1(\alpha,\beta)$ is conjugate to $U_1(0,\rho)$ by Proposition \ref{prop34}. Now, it is quite easy to 
see that
two subgroups $U_1(0,\rho_1)$  and $U_1(0,\rho_2)$ ($\rho_1,\rho_2\in \F^\square$) are conjugate in $\AGL_4(\F)$ if and 
only if $\rho_1=\rho_2$.

If $\alpha^2+4\beta=0$, we consider the epimorphism 
$\Psi: \F[t_1, t_2]\rightarrow \A_{U_1(\alpha,\beta)}$ defined by
\begin{equation*}
\Psi(t_1)  =   X_1,\qquad \Psi(t_2)  =  \alpha X_1- 2X_2, 
\end{equation*}
whose kernel is $\Ker(\Psi)=\langle t_1^3,\; t_1^2 t_2,\; t_2^2\rangle$.
Again by Proposition \ref{prop34}, $U_1(\alpha,\beta)$ is conjugate to $U_1(0,0)$ (which is not conjugate to 
$U_1(0,\rho)$, $\rho \in \F^\square$). 

Next, suppose that $\char \F=2$. If $\alpha\neq 0$, take $\epsilon \in \H(\F)$ such that 
$\epsilon+\frac{\beta}{\alpha^2}\in \G(\F)$.  Then, there exists $r\in \F$ such that 
$r^2+\alpha r+ \beta+\alpha^2\epsilon=0$ (a polynomial $t^2+ t +\lambda$ is reducible  in $\F[t]$ if and only if 
$\lambda \in \G(\F)$). In this case,
we can define an epimorphism
$\Psi: \F[t_1, t_2]\rightarrow \A_{U_1(\alpha,\beta)}$ by setting 
$$\Psi(t_1)=rX_1+X_2,\qquad \Psi(t_2)=\alpha X_1.$$
We obtain 
$\Ker(\Psi)=\langle t_1^3,\; t_2^3,\; t_1^2+t_1t_2+\epsilon t_2^2\rangle$, proving that $U_1(\alpha,\beta)$ is 
conjugate to $U_1(1,\epsilon)$.
Now, $U_1(1,\epsilon_1)$ and $U_1(1,\epsilon_2)$ are conjugate in $\AGL_4(\F)$ 
if and only if $\epsilon_1+\epsilon_2\in \G(\F)$.
It follows that $U(\alpha,\beta)$ is conjugate to exactly one subgroup $U_1(1,\epsilon)$ with $\epsilon \in \H(\F)$.

Finally, assume $\alpha=0$. For $\beta=0$ observe that $U_1(0,0)^g=U_1(0,1)$, where $g=I_5+E_{3,2}+E_{5,4}$. If 
$\beta\neq 0$, write $\beta=\omega^2\rho$ for some $\omega \in \F^\ast$ and a unique $\rho\in \F^\square$.
An epimorphism
$\Psi: \F[t_1, t_2]\rightarrow \A_{U_1(0,\beta)}$ can be defined by taking 
$$\Psi(t_1)=X_2,\qquad \Psi(t_2)=\omega X_1.$$
We have  $\Ker(\Psi)=\langle t_1^3,\; t_2^3,\; t_1^2+\rho t_2^2\rangle$ and so $U_1(0,\beta)$ is conjugate to 
$U_1(0,\rho)$.
Furthermore, given $\rho_1,\rho_2\in \F^\square$, the subgroup $U_1(0,\rho_1)$ is conjugate to $U_1(0,\rho_2)$ if and 
only if $\rho_1\star\rho_2$, that is if and only if $\rho_2=\frac{x_1^2+x_2^2\rho_1}{x_3^2+x_4^2\rho_1}$ for some $x_1,x_2,
x_3, x_4 \in \F$ such that $x_1x_4\neq x_2x_3$ (see the notation described in the Introduction).

To conclude we observe that subgroups $U_1(0,\rho)$, $\rho \in \K(\F)$, are not conjugate to subgroups 
$U_1(1,\epsilon)$, $\epsilon \in \H(\F)$.
\end{proof}

We can now classify the abelian regular subgroups of $\AGL_4(\F)$.

\begin{theorem}\label{main_ab}
Let $\F$ be a field. The distinct conjugacy classes of abelian regular subgroups of $\AGL_4(\F)$ can be represented by
the subgroups described in the first column of Table \ref{tab1}.
\end{theorem}

\begin{proof}
Let $R\leq \AGL_4(\F)$ be an abelian regular subgroup. 
Clearly, $2\leq \d(R)\leq 5$. As seen in Section \ref{sec2}, if $\d(R)=5$ then 
$R$ is conjugate to $S_{(5)}$, and if $\d(R)=4$  then $R$ is conjugate either to $S_{(4,1)}$ or to $S_{(3,2)}^\s$.

Suppose that $\d(R)=3$. Then $2\leq \r(R)\leq 3$.
If $\r(R)=3$, we apply Lemma \ref{J3J2}: if $\char \F\neq 2$, $R$ is conjugate to 
$U_1(0,\rho)$ for a unique $\rho\in \F^\square\cup \{0\}$;
if $\char \F=2$, it is conjugate either to 
$U_1(1,\epsilon)$ for a unique $\epsilon \in \H(\F)$ or to $U_1(0,\rho)$ for a unique $\rho\in \K(\F)$.
If $\r(R)=2$, then $R$ is conjugate to a subgroup $\RR_{D}$ for some $D\in \Mat_3(\F)$.

Suppose $\d(R)=2$. Then $1\leq \r(R)\leq 2$. If $\r(R)=2$ we apply 
\cite[Lemma 7.7]{PT}, obtaining that $\char\F=2$ and $R$ is 
conjugate to $U_2(0,0,0,0,0)=\RR_D$, where
$D=E_{1,2}+E_{2,1}$.
If $\r(R)=1$, then $U=\Tr_4=\RR_{0}$, where $0\in \Mat_3(\F)$ is the zero matrix.  

Notice that in three cases, $R$ is conjugate to a subgroup $\RR_D$. 
By Lemma \ref{RD} the statement of the theorem is proved considering a set $\Pi_S(\F)$ of representatives for 
projective congruent 
classes of symmetric matrices in $\Mat_3(\F)$.
\end{proof}

We now provide explicit representatives for algebraically closed fields, for $\R$  and for finite fields.

\begin{cor}\label{ab_ac}
Let $\F=\overline\F$ be an algebraically closed field. If $\char \F\neq 2$, there are exactly $9$ distinct conjugacy 
classes of abelian regular subgroups of $\AGL_4(\F)$, that can be represented by
$$S_{(5)},\; S_{(4,1)},\; S_{(3,2)}^\s, \;U_1(0,0),\;U_1(0,1), \;\RR_{E_{1,1}},\; 
\RR_{E_{1,1}+E_{2,2}},\; \RR_{I_3},\; \Tr_4.$$
If $\char \F=2$, there are exactly $10$ distinct conjugacy classes of abelian regular subgroups of $\AGL_4(\F)$, that 
can be represented by
$$S_{(5)},\; S_{(4,1)},\;  S_{(3,2)}^\s,\; U_1(0,0),\; U_1(1,0),\;\RR_{E_{1,1}},\; \RR_{E_{1,1}+E_{2,2}},\;
\RR_{E_{2,3}+E_{3,2}},\; \RR_{I_3},\; \Tr_4.$$
\end{cor}

\begin{proof}
Since $\F$ is algebraically closed, $\F^\square=\{1\}$ and, when $\char \F=2$, $\H(\F)=\{0\}$ and $\K(\F)=\{1\}$. 
The set $\Pi_S(\F)$ can be obtained applying, for instance, \cite[Theorem 2.1]{HS}. Hence, the statement is proved noticing that, 
as observed in the proof of Lemma \ref{J3J2}, when $\char \F=2$ the subgroups $U_1(0,0)$ and $U_1(0,1)$ are conjugate.
\end{proof}

\begin{cor}\label{ab_real}
There are exactly $12$ distinct conjugacy 
classes of abelian regular subgroups of $\AGL_4(\R)$, that can be represented by
$$S_{(5)},\;\; S_{(4,1)},\;\; S_{(3,2)}^\s,\; \;U_1(0,0),\;\;U_1(0,1),\;\; U_1(0,-1),$$
$$\RR_{E_{1,1}},\;\;\RR_{E_{1,1}+E_{2,2}},\;\;\RR_{E_{1,1}-E_{2,2}},\;\;\RR_{E_{1,1}+E_{2,2}-E_{3,3}},\;\; \RR_{I_3},\;\; \Tr_4.$$
\end{cor}

\begin{proof}
First of all, we can take $\R^\square=\{1,-1\}$. 
The set $\Pi_S(\F)$ can be obtained again applying  \cite[Theorem 2.1]{HS}, noticing that two matrices $A,B\in 
\Mat_n(\F)$ are  projectively congruent if and only if $A$ is congruent to $\pm B$.
\end{proof}

\begin{cor}\label{ab_f}
Let $\F=\F_q$ be a finite field.
There are exactly $11$ distinct conjugacy classes of abelian regular subgroups of $\AGL_4(q)$, that can be represented by:
\begin{itemize}
\item[(a)] for $q$  odd,
$$S_{(5)},\;\; S_{(4,1)},\;\;  S_{(3,2)}^\s,\;\; U_1(0,0),\;\; U_1(0,1),\;\; 
U_1(0 ,\xi),\;\;\RR_D,$$
where $x^2-\xi$ is a fixed irreducible polynomial of $\F_q[x]$ and
$$D\in \Pi_S(\F_q)=\{0,\; E_{1,1},\; E_{1,1}+E_{2,2},\;E_{1,1}+\xi E_{2,2},\; I_3\};$$
\item[(b)] for $q$  even, 
$$S_{(5)},\;\; S_{4,1},\;\;  S_{(3,2)}^\s,\;\; U_1(0,0),\;\; U_1(1,0),\;\;
U_1(1,\xi),\;\; \RR_D,$$
where $x^2+x+\xi$ is a fixed irreducible polynomial of $\F_q[x]$ and
 $$D\in \Pi_S(\F_q)=\{0,\; E_{1,1},\; E_{1,1}+E_{2,2},\; E_{2,3}+E_{3,2},\; I_3\}.$$
\end{itemize}
\end{cor}

\begin{proof}
For $q$ odd we have $\F^\square=\{1,\xi\}$ and for $q$ even we have $\F^\square=\K(\F)=\{1\}$ and $\H(\F)=\{0,\xi\}$.
The set $\Pi_S(\F_q)$ has been determined in \cite[Theorem 4]{W1}.
\end{proof}

\section{The nonabelian case}\label{nonab}

In this section we assume that $R\in \D$ is nonabelian. We are reduced to study the cases
when $(\d(\Z(R)),\r(\Z(R)))=(3,2)$, $(2,2)$ or $(2,1)$.
We recall that if $(\d(\Z(R)),\r(\Z(R)))=(3,2)$ then $R$ is conjugate to a subgroup of shape $\RR_D$.

If $\d(\Z(R))=\r(\Z(R))=2$, then $R$ is conjugate to $U_2(\alpha_1,\alpha_3,\gamma_1,\gamma_3,\zeta)$ as in \eqref{U2}. 
Clearly, $U_2(\alpha_1,\alpha_3,\gamma_1,\gamma_3,0)=\RR_D$, where
$$D=\begin{pmatrix}
\alpha_1 & 1 & \alpha_3 \\ 1 & 0 & 0 \\ \gamma_1 & 0 & \gamma_3     
    \end{pmatrix}.$$
So, we have to consider now the case $\zeta\neq 0$. Notice that since $R$ is nonabelian, we must have $\gamma_1\neq 
\alpha_3$. Furthermore, we have $\d(R)=4$ and $\r(R)=3$.

\begin{lemma}\label{lemU2}
There are exactly $2$ conjugacy classes of nonabelian regular subgroups $R\in \D$ such that $\d(\Z(R))=\r(\Z(R))=2$, 
not conjugate to $\RR_A$ for any $A\in \Mat_3(\F)$. 
Such classes can be represented by the subgroups $U_2(0,1,0,1,1)$ and $U_2(0,1,0,0,1)$. 
\end{lemma}

\begin{proof}
By the previous considerations, we may assume that  $R=U_2(\alpha_1,\alpha_3,\gamma_1,\gamma_3,\zeta)$ where $\zeta\neq 0$ 
and $\gamma_1\neq \alpha_3$.

First, suppose $\gamma_3\neq 0$: in this case $\k(R)=1$. Consider the algebra $\A=\A_{U_2(0,1,0,1,1)}$ defined in \eqref{LR}:
$$\A=\sp(t_1,t_2)\; \textrm{ such that: }\; t_1^3=t_2^2=t_1t_2; \quad t_2t_1=0.$$
The function $\Psi: \A \rightarrow \A_R$ defined by
$$\Psi(t_1)=\frac{(\alpha_3-\gamma_1)^2}{\gamma_3\zeta }X_1,\qquad
\Psi(t_2)=\frac{-\gamma_1 (\alpha_3-\gamma_1)^3}{\gamma_3^2\zeta }X_2+\frac{(\alpha_3-\gamma_1)^3}{\gamma_3^2\zeta}X_3,$$
is an isomorphism: by Proposition \ref{prop34}, $R$ is conjugate to $U_2(0,1,0,1,1)$. 

Suppose now that $\gamma_3= 0$: in this case $\k(R)=2$ and so $R$ is not conjugate to $U_2(0,1,0,1,1)$. Consider the algebra $\A=\A_{U_2(0,1,0,0,1)}$:
$$\A=\sp(t_1,t_2)\; \textrm{ such that: }\; t_1^3=t_1t_2; \quad t_2^2=t_2t_1=0.$$
The function $\Psi: \A \rightarrow \A_R$ defined by
$$\Psi(t_1)=(\alpha_3-\gamma_1)X_1,\qquad \Psi(t_2)=-\zeta\gamma_1(\alpha_3-\gamma_1) X_2+ \zeta(\alpha_3-\gamma_1)X_3$$
is an isomorphism and so $R$ is conjugate to  $U_2(0,1,0,0,1)$.
\end{proof}

We now consider the case when $\d(\Z(R))=2$ and $\r(\Z(R))=1$. Instead of working in $\C_{\AGL_4(\F)}(\diag(J_2,I_3))$, it is more convenient to write $R$ as $\RR(\tilde R,D)$, for some $\tilde R\in \Delta_3(\F)$ and 
$D\in \Mat_3(\F)$, see \eqref{Rtilde}. The regular subgroups in $\Delta_3(\F)$ have been classified in \cite{PT}.

\begin{theorem}\cite[Section 7.2]{PT}
Let $\F$ be a field.
The distinct conjugacy classes of regular subgroups in $\Delta_3(\F)$ are represented by
$$S_{(4)},\quad S_{(3,1)},\quad  R_\rho\;\; (\rho \in \F^\square),\quad \Tr_3,\quad N_1, \quad 
N_{3,\lambda}\;\;(\lambda \in \F^\ast),$$
$$U_1^3\;\;(\textrm{if } \char \F=2),\quad N_2\;\;(\textrm{if } \char \F\neq 2).$$
\end{theorem}

Suppose that $R=\RR(S_{(4)},D)$ for some $D\in \Mat_3(\F)$. By \eqref{dt}, $D=\alpha_1E_{1,1}+\alpha_2 (E_{1,2}
E_{2,1})+\alpha_3(E_{1,3}+E_{2,2}+E_{3,1})$ for some $\alpha_1,\alpha_2,\alpha_3\in \F$. It follows that $R$ is abelian.
Furthermore, $\RR(\Tr_3,D)=\RR_D$. Hence, we have to consider the six subcases corresponding to $\tilde R =S_{(3,1)}, R_\rho, U_1^3, N_1, N_2, N_{3,\lambda}$.

\begin{lemma}\label{lemU3}
There are exactly $|\F|+1$ distinct conjugacy classes of nonabelian subgroups
$R=\RR(S_{(3,1)},D)\in \D$  such that  $\d(\Z(R))=2$ and $\r(\Z(R))=1$. 
Such classes can be represented by the subgroups corresponding to 
$$
\begin{array}{cccc}
D & \d(R) & \r(R)& k(R)\\ \hline 
E_{3,1} &  3 & 3 & 3\\
E_{3,1}+E_{3,3} &  3 &  3  & 2\\
E_{1,3}+\lambda E_{3,1}\;\;(\lambda \in \F\setminus \{0,1\}) & 3 & 3 & 2\\
E_{1,3} & 3 & 2 & 2
\end{array}$$
These subgroups are not conjugate to $\RR_A$ for any $A\in \Mat_3(\F)$.
\end{lemma}

\begin{proof}
Let $R=\RR(S_{(3,1)},D)$ for some $D\in \Mat_3(\F)$ such that $\d(\Z(R))=2$ and $\r(\Z(R))=1$. 
By \eqref{dt} and the hypothesis that $R$ is nonabelian,  we obtain $D=\begin{pmatrix} \bar\alpha_1 & \bar\alpha_2 & \bar\alpha_3\\
\bar\alpha_2 & 0 & 0 \\\bar\gamma_1 & 0 & \bar\gamma_3 \end{pmatrix}$ with $\bar\gamma_1\neq \bar\alpha_3$.
Moreover, the condition $\r(\Z(R))=1$ implies $\bar\alpha_2=0$.
Then,  $R$ is conjugate via $g=I_5-\bar\alpha_1 E_{3,5}$ to 
$$U_3(\alpha_3,\gamma_1,\gamma_3)=\begin{pmatrix}
1 & x_1 & x_2 & x_3 & x_4\\
0 & 1 & x_1 & 0 & \alpha_3 x_3\\
0 & 0 & 1 & 0 & 0 \\
0 & 0 & 0 & 1 & \gamma_1 x_1 +\gamma_3 x_3\\
0 & 0 & 0 & 0 & 1
\end{pmatrix},
$$
for some $\alpha_3,\gamma_1,\gamma_3\in \F$ such that $\gamma_1\neq \alpha_3$. 

Suppose that $\gamma_3\neq 0$. Then $\d(R)=\r(R)=3$ and $\k(R)=2$.
Consider the algebra $\A=\A_{U_3(0,1,1)}$:
\begin{equation*}\label{U3(0,1,1)}
\A=\sp(t_1,t_2)\; \textrm{ such that: }\; t_2 t_1 = t_2^2; \quad t_1^3=t_1t_2=0.
\end{equation*}
The function $\Psi: \A \rightarrow \A_R$ defined by
$$\Psi(t_1)=\gamma_3 X_1- \alpha_3 X_3,\qquad \Psi(t_2)=(\gamma_1-\alpha_3)X_3$$
is an isomorphism and so $R$ is conjugate to  $U_3(0,1,1)$.
 
Suppose now that $\gamma_3=\alpha_3=0$ (and so $\gamma_1\ne 0$). Then 
$\d(R)=\r(R)=\k(R)=3$.  In this case, consider the  algebra $\A=\A_{U_3(0,1,0)}$:
\begin{equation*}\label{U3(0,1,0)}
\A=\sp(t_1,t_2)\; \textrm{ such that: }\; t_1^3=t_2^2=t_1t_2=t_2t_1^2=0.
\end{equation*}
The function $\Psi: \A \rightarrow \A_R$ defined by
$$\Psi(t_1)=X_1,\qquad \Psi(t_2)=X_3$$
is an isomorphism: it follows that $R$ is conjugate to  $U_3(0,1,0)$.

Finally, suppose that $\gamma_3=0$ and $\alpha_3\neq 0$. For any $\lambda \in \F\setminus\{1\}$, consider the algebra $\A_\lambda=\A_{U_3(1,\lambda,0)}$:
\begin{equation*}\label{U3(1,lambda,0)}
\A_\lambda=\sp(t_1,t_2)\; \textrm{ such that: }\; t_1^3=t_2^2=t_1^2t_2=0; \quad t_2t_1=\lambda t_1t_2.
\end{equation*}
The function $\Psi: \A_\lambda \rightarrow \A_R$ defined by
$$\Psi(t_1)=X_1,\qquad \Psi(t_2)=X_3,\quad \lambda=\frac{\gamma_1}{\alpha_3}$$
is an isomorphism and hence $R$ is conjugate to  $U_3(1,\lambda,0)$ for some $\lambda\neq 1$.
Now, notice that if $\lambda=0$ then $\d(U_3(1,0,0))=3$ and $\r(U_3(1,0,0))=\k(U_3(1,0,0))=2$. Otherwise, $\d(U_3(1,\lambda,0))=\r(U_3(1,\lambda,0))=3$ and $\k(U_3(1,\lambda,0))=2$. 
Furthermore, $U_3(1,0,0)$ is not conjugate to any $\RR_A$ (direct computations) and the same holds for the other subgroups, since $\r(\RR_A)\leq 2$. The subgroups $U_3(1,\lambda_1,0)$ and $U_3(1,\lambda_2,0)$ are conjugate if and only if $\lambda_1=\lambda_2$. To conclude, we observe that $U_3(1,\lambda,0)$ is not conjugate to $U_3(0,1,1)$.
\end{proof}

Before to consider the next subcase, we recall the following notation given in the Introduction.
For any $\rho \in 
\F^\square$,  we denote by $\CC_\rho(\F)$ a fixed set of representatives for the equivalence classes with respect to 
the following relation  $\rel_\rho$ defined on $\CCC(\F)=\left(\F\setminus\{1\}\right)\times \F$:
given $(\beta_1,\beta_2),(\beta_3,\beta_4)\in \CCC(\F)$, we write $(\beta_1,\beta_2)\rel_\rho (\beta_3,\beta_4)$ if and 
only if 
either $(\beta_3,\beta_4)=(\beta_1,\pm \beta_2)$ or
\begin{equation}\label{b3b4}
\beta_3=\frac{\beta_1 t^2+ \beta_2 t-\rho}{t^2 + \beta_2 t -\rho\beta_1}\quad\textrm{ and }\quad
 \beta_4= \pm \frac{\beta_2 t^2-2\rho(\beta_1+1)t-\rho\beta_2}
  {t^2+\beta_2 t -\rho\beta_1}
\end{equation}
for some $t \in \F$  such that $(t^2+\rho )(t^2 + \beta_2 t -\rho\beta_1)\neq 0$.

\begin{lemma}\label{U4}
The distinct conjugacy classes of nonabelian subgroups
$R=\RR(R_\rho,D)$ $\in \D$  such that  $\d(\Z(R))=2$ and $\r(\Z(R))=1$
can be represented by 
$$\RR(R_\rho,E_{1,2}+\beta_1 
E_{2,1}+\beta_2 E_{2,2})$$
with $\rho \in \F^\square$ and $(\beta_1,\beta_2)\in \CC_\rho(\F)$.
Furthermore,  $\d(R)=\r(R)=3$ and $\k(R)=2$.
\end{lemma}

\begin{proof}
Let $R=\RR(R_\rho,D)\in \D$, with $\rho \in \F^\square$ and $D\in \Mat_3(\F)$,  be a nonabelian subgroup  such that 
$\d(\Z(R))=2$ and $\r(\Z(R))=1$.
By \eqref{dt} and the hypothesis that $R$ is nonabelian, we have $D=\begin{pmatrix}\bar \alpha_1 & \bar\alpha_2 & 0\\
\bar\beta_1 & \bar\beta_2 & 0 \\ 0 & 0 & 0 \end{pmatrix}$, with
$\bar\beta_1\neq \bar\alpha_2$.
Taking $g=\diag\left(I_3,\begin{pmatrix} 1 & -\bar\alpha_1 \bar\alpha_2^{-1}\\ 0  & 
\bar\alpha_2^{-1}\end{pmatrix}\right)$ if $\bar \alpha_2\neq 0$ and 
$g=\diag\left(1,\begin{pmatrix}
0 & 1 \\ \rho & 0\end{pmatrix}, \begin{pmatrix} \rho & -\bar\beta_2\bar\beta_1^{-1} \\
0 & \rho\bar\beta_1^{-1}\end{pmatrix}\right)$ if $\bar\alpha_2=0$,
we obtain that $R$ is conjugate via $g$ to the 
subgroup
$$U_4(\rho,\beta_1,\beta_2)=\begin{pmatrix}
1 & x_1 & x_2 & x_3 & x_4\\
0 & 1 & 0 & x_1 &  x_2\\
0 & 0 & 1 & \rho x_2 & \beta_1 x_1+\beta_2 x_2\\
0 & 0 & 0 & 1 & 0\\
0 & 0 & 0 & 0 & 1
\end{pmatrix},\quad \textrm{ where } \beta_1\neq 1.$$ 

Let $\rho_1,\rho_2\in \F^\square$, $\beta_1,\beta_2,\beta_3,\beta_4\in \F$ with $\beta_1,\beta_3\neq 1$.
It is quite easy to verify that the subgroups $U_4(\rho_1,\beta_1,\beta_2)$ and $U_4(\rho_2,\beta_3,\beta_4)$ are conjugate in 
$\AGL_4(\F)$ if and only if there exists an element $g=\diag(1,A,B)\in \GL_5(\F)$, with $A=\begin{pmatrix} a_1 & a_2 
\\ 
a_3 & 
a_4\end{pmatrix}$, 
$B=\begin{pmatrix} b_1 &b_2 \\b_3 & b_4\end{pmatrix}$, such that 
$U_4(\rho_1,\beta_1,\beta_2)^g=U_4(\rho_2,\beta_3,\beta_4)$. This holds if and only if
$b_1  =  \rho_2 a_2^2+a_1^2$,
$b_2  =  (\beta_3+1) a_1 a_2+\beta_4 a_2^2$,
$b_3  =  \rho_2 a_2 a_4+a_1a_3$, $b_4  =  \beta_3 a_2 a_3+\beta_4 a_2 a_4+a_1 a_4$
and
\begin{equation}\label{eq1}
a_1 a_3 + a_2 a_4 \rho_2   =  0,
\end{equation}
\begin{equation}\label{eq2}
\left\{
\begin{array}{c}
(\beta_1\beta_3-1) a_2 a_3+\beta_4(\beta_1-1) a_2a_4+(\beta_1-\beta_3)a_1a_4 =  0,\\
\rho_1\rho_2a_2^2+ \rho_1a_1^2- \rho_2a_4^2-a_3^2  =  0,\\
\beta_2(\beta_3 a_2a_3+\beta_4a_2a_4 +a_1a_4)
+ \beta_4(a_2^2\rho_1-a_4^2)+(\beta_3+1)(a_1a_2\rho_1-a_3a_4)=0
\end{array}\right.
\end{equation}
provided that
\begin{equation}\label{det}
(a_1a_4 - a_2a_3)(a_1^2 + a_1a_2\beta_4 - a_2^2\rho_2\beta_3)\neq 0.
\end{equation}
Now, if $a_2=0$, from \eqref{eq1} and \eqref{det} we get  $a_3=0$ and in this case  \eqref{eq2} gives
$\beta_3=\beta_1$, $\rho_2=\rho_1$ and $\beta_4=\pm \beta_2$
(notice that, taking $g=\diag(1,1,-1,1,-1)$, we have $U_4(\rho_1,\beta_1,\beta_2)^g=U_4(\rho_1,\beta_1,-\beta_2)$).
Hence, we may suppose $a_2\neq 0$ and $a_4=-a_1a_3(a_2\rho_2)^{-1}$ by \eqref{eq1}.
In this case \eqref{eq2} gives $\rho_2=\rho_1$, 
$a_3=\pm a_2\rho_1$, $f_0=0$ and $f_+=0$ or $f_{-}=0$, where
\begin{eqnarray*}
f_0 & = &  a_1^2(\beta_1 - \beta_3)  + a_1a_2 \beta_4 (\beta_1 -1)- a_2^2\rho_1 (
\beta_1\beta_3-1),\\
f_{\pm} & = & a_1^2(\beta_2\pm \beta_4)+a_1a_2(\beta_2\beta_4\mp 2\rho_1(\beta_3+1))-a_2^2\rho_1(\beta_2\beta_3\pm 
\beta_4)
\end{eqnarray*}
provided that 
\begin{equation}
(a_1^2 +a_2^2\rho_1)(a_1^2 + a_1a_2\beta_4 - a_2^2\rho_1\beta_3)\neq 0.
\end{equation}
Taking $a_2=\pm 1$, it follows that  $U_4(\rho_1,\beta_1,\beta_2)$ is conjugate to $U_4(\rho_2,\beta_3,\beta_4)$ if  
and only if $\rho_2=\rho_1$ and  either $(\beta_3,\beta_4)=(\beta_1,\pm \beta_2)$ or
$$(\beta_3,\beta_4)=\left(\frac{\beta_1 t^2+ \beta_2 t-\rho_1}{t^2 + \beta_2 t -\rho_1\beta_1},\;
  \pm \frac{\beta_2 t^2-2\rho_1(\beta_1+1)t-\rho_1\beta_2}
  {t^2+\beta_2 t -\rho_1\beta_1}\right)$$
for some $t \in \F$  such that $(t^2+\rho)(t^2 + \beta_2 t -\rho_1\beta_1)\neq 0$. In other words, this holds if and 
only if
$\rho_2=\rho_1$ and $(\beta_1,\beta_2)\rel_{\rho_1}(\beta_3,\beta_4)$.
\end{proof}

\begin{rem}\label{rem}
Let $\rho \in \F^\square$, $\beta_1,\beta_3\in \F\setminus\{1\}$ and $\beta_2,\beta_4\in \F$.
Then:
\begin{itemize}
\item[(a)] $U_4(\rho,\beta_1,\beta_2)$ and $U_4(\rho,-1,0)$ are conjugate if and only if $(\beta_1,\beta_2)=(-1,0)$;
\item[(b)] $U_4(\rho,\beta_1,\beta_2)$ and $U_4(\rho,\beta_1,\beta_4)$ are conjugate if and only if $\beta_4=\pm \beta_2$;
\item[(c)] $U_4(\rho,\beta_1,0)$ and $U_4(\rho,\beta_3,0)$ are conjugate if and only if either $\beta_3=\beta_1$ or $\beta_1\neq 0$ and $\beta_3=\beta_1^{-1}$;
\item[(d)]  $U_4(\rho,\beta_1,0)$ and $U_4(\rho,0,\beta_4)$ are conjugate if and only if $\beta_1=\rho\omega^2$ and $\beta_4=\frac{2\rho\omega}{\rho\omega^2-1}$ for some $\omega\in \F$ such that $\rho \omega^2\neq \pm 1$.
\end{itemize}
\end{rem}

\begin{rem} It will be useful to give the presentation of the following algebras:
\begin{itemize}
\item[(a)] $\A_{U_4(\rho,0,\lambda)}=\sp(t_1,t_2)\; \textrm{ such that: }\;
t_1^3=t_2t_1=0; \quad  t_2^2=\rho t_1^2+\lambda t_1t_2$;
\item[(b)] $\A_{U_4(\rho,-1,0)}=\sp(t_1,t_2)\; \textrm{ such that: }\;
t_1^3=t_1t_2+t_2t_1=0; \quad  t_2^2=\rho t_1^2$.
\end{itemize}
\end{rem}

\begin{lemma}\label{lemU5}
Assume $\char \F=2$. There are exactly $|\H(\F)|$ distinct conjugacy classes of nonabelian subgroups
$R=\RR(U_1^3,D)\in \D$  such that  $\d(\Z(R))=2$ and $\r(\Z(R))=1$. 
Such classes can be represented by the subgroups corresponding to 
$$D=E_{1,1}+E_{1,2}+\lambda E_{2,2}\;\; (\lambda \in \H(\F)).$$
Furthermore, $\d(R)=\r(R)=3$ and $\k(R)=2$.
\end{lemma}

\begin{proof}
Let $R=\RR(U_1^3,D)$ for some $D\in \Mat_3(\F)$.
By \eqref{dt} and the hypothesis that $R$ is nonabelian, we have $D=\begin{pmatrix}\bar \alpha_1 & \bar\alpha_2 & 0\\
\bar\beta_1 & \bar\beta_2 & 0 \\ 0 & 0 & 0 \end{pmatrix}$ with $\bar\beta_1\neq\bar \alpha_2$.
Then, $R$ is conjugate via $g=I_5+\bar\beta_1 E_{4,5}$ to 
$$U_5(\alpha_1,\alpha_2,\beta_2)=\begin{pmatrix}
1 & x_1 & x_2 & x_3 & x_4\\
0 & 1 & 0 & x_2 & \alpha_1 x_1+\alpha_2 x_2\\
0 & 0 & 1 & x_1 & \beta_2 x_2\\
0 & 0 & 0 & 1 & 0\\
0 & 0 & 0 & 0 & 1
\end{pmatrix},
$$
where $\alpha_2\neq 0$. Then $\d(R)=\r(R)=3$ and $\k(U)=2$.

For any $\lambda \in\F$, consider the algebra $\A_\lambda=\A_{U_5(1,1,\lambda)}$:
$$\A_\lambda=\sp(t_1,t_2)\; \textrm{ such that: }\;
t_1^3=t_2^3=t_1t_2+t_2t_1+t_1^2=0; \quad  t_2^2=\lambda t_1^2
$$
(clearly, when $\lambda=0$, the condition $t_2^3=0$ can be omitted). We now define an isomorphism of split local algebras $\Psi: \A_\lambda \rightarrow \A_R$
in the following way:
$$\left\{
\begin{array}{lll}
\Psi(t_1)=\alpha_2 X_1+(\alpha_1+\alpha_2) X_2, & \Psi(t_2)=\alpha_2 X_2,\; & \textrm{ if } \beta_2=0,\\
\Psi(t_1)=\alpha_2 X_2, & \Psi(t_2)=\beta_2 X_1,\; & \textrm{ if } \beta_2\neq 0,
\end{array}\right.$$
where $\lambda=\frac{\alpha_1 \beta_2 }{\alpha_2^2}$.
Hence, $R$ is conjugate to  $U_5(1,1,\lambda)$ for some $\lambda \in \F$.
Now, $U_5(1,1,\lambda_1)$ and $U_5(1,1,\lambda_2)$ are conjugate if and only if $\lambda_1+\lambda_2 \in \G(\F)$. We 
conclude that are exactly $|\H(\F)|$ conjugacy classes.
\end{proof}

\begin{lemma}\label{lemU6}
Let $R=\RR(N_1,D)\in \D$ be a nonabelian subgroup such that  $\d(\Z(R))=2$ and $\r(\Z(R))=1$. Then $R$ is conjugate to  one of the following subgroups:
$\RR_{E_{2,1}}$, $U_3(1,0,0)$, $U_3(0,1,0)$, $U_3(0,1,1)$, $U_4(\rho,-1,0)$, $U_4(\rho,0,\lambda)$ 
or $U_5(1,1,0)$, for some $\rho\in \F^\square$ and $\lambda\in \F$.
\end{lemma}

\begin{proof}
Let $R=\RR(N_1,D)$ for some $D\in \Mat_3(\F)$.
By \eqref{dt}, $D=\begin{pmatrix} \bar\alpha_1 & \bar\alpha_2 & 0\\
\bar\beta_1 & \bar\beta_2 & 0 \\ 0 & 0 & 0 \end{pmatrix}$. 
Then, $R$ is conjugate via $g=I_5-\bar\beta_1 E_{4,5}$ to 
$$U_6(\alpha_1,\alpha_2,\beta_2)=\begin{pmatrix}
1 & x_1 & x_2 & x_3 & x_4\\
0 & 1 & 0 & 0 & \alpha_1 x_1+\alpha_2 x_2\\
0 & 0 & 1 & x_1 & \beta_2 x_2\\
0 & 0 & 0 & 1 & 0\\
0 & 0 & 0 & 0 & 1
    \end{pmatrix},
$$
for some $\alpha_1,\alpha_2,\beta_2\in \F$.
It is easy to see that $U_6(\alpha_1,\alpha_2,\beta_2)$ is  conjugate to some $\RR_A$ if and only if 
$\alpha_1=\alpha_2=\beta_2=0$. More  in detail, $U_6(0,0,0)^g=\RR_{E_{2,1}}$,
where $g=\diag(I_3,E_{1,2}+E_{2,1})$.

If $\alpha_1=\alpha_2=0$ and $\beta_2\neq 0$, then $\d(R)=3$ and $\r(R)=\k(R)=2$. 
The function $\Psi: \A_{U_3(1,0,0)} \rightarrow \A_R$ defined by
$$\Psi(t_1)=X_2,\qquad \Psi(t_2)=X_1$$
is an isomorphism. By Proposition \ref{prop34},  $R$ is conjugate to the subgroup $U_3(1,0,0)$.

If $\alpha_1\neq 0$ and $\alpha_2=\beta_2=0$, then $\d(R)=\r(R)=\k(R)=3$. 
The function $\Psi: \A_{U_3(0,1,0)} \rightarrow \A_R$ defined by
$$\Psi(t_1)=X_1,\qquad \Psi(t_2)= X_2$$
is an isomorphism and so $R$ is conjugate to  $U_3(0,1,0)$.

For the other values of $\alpha_1,\alpha_2,\beta_2$ we have $\d(R)=\r(R)=3$ and $\k(R)=2$.
If $\alpha_1,\beta_2\neq 0$ and $\alpha_2=0$, write $\frac{\alpha_1}{\beta_2}=\omega^2\rho$ for some $\omega \in \F^\ast$ and a unique $\rho \in \F^\square$:
the function $\Psi: \A_{U_4(\rho,0,0)} \rightarrow \A_R$ defined by
$$\Psi(t_1)=\omega X_2,\qquad \Psi(t_2)= X_1$$
is an isomorphism. Hence, $R$ is conjugate to $U_4(\rho,0,0)$.

We are left to consider the case where $\alpha_2\neq 0$. Suppose $\alpha_1=\beta_2=0$.
If $\char \F\neq 2$, write $-1=\omega^2 \rho$ for some $\omega \in \F^\ast$ and a unique $\rho \in \F^\square$:
the function $\Psi: \A_{U_4(\rho,-1,0)} \rightarrow \A_R$ defined by
$$\Psi(t_1)=\omega X_1+\omega X_2,\qquad \Psi(t_2)=X_1- X_2$$
is an isomorphism  and hence $R$ is conjugate to  $U_4(\rho,-1,0)$.
If $\char \F =2$, 
the function $\Psi: \A_{U_5(1,1,0)} \rightarrow \A_R$ defined by
$$\Psi(t_1)=X_1+X_2,\qquad \Psi(t_2)= X_2$$
is an isomorphism  and hence $R$ is conjugate to  $U_5(1,1,0)$.

If $(\alpha_1,\beta_2)\neq (0,0)$,  set  $\Delta=\alpha_1\beta_2-\alpha_2^2$. If $\Delta=0$ (and so $\beta_2\neq 0$) then the function $\Psi: \A_{U_3(0,1,1)} \rightarrow \A_R$ defined by
$$\Psi(t_1)=\beta_2 X_1,\qquad \Psi(t_2)=\beta_2 X_1-\alpha_2 X_2$$
is an isomorphism. Hence, $U$ is conjugate to  $U_3(0,1,1)$.
If $\Delta\neq 0$, write $\frac{\Delta}{\alpha_2^2}=\omega^2 \rho$ for some $\omega\in \F^\ast$ and a unique $\rho \in \F^\square$. The function $\Psi: \A_{U_4(\rho,0,\omega^{-1})} \rightarrow \A_R$ defined by
$$\left\{
\begin{array}{lll}
\Psi(t_1)=\alpha_2\omega X_2,& 
\Psi(t_2)=-\beta_2 X_1+\alpha_2 X_2 & \textrm{ if }\beta_2\neq 0,\\
\Psi(t_1)=\alpha_2\omega X_1-\alpha_1\omega X_2,& 
\Psi(t_2)=\alpha_2 X_1 & \textrm{ if }\beta_2= 0,\\
\end{array}\right.$$
is an isomorphism. Hence, $R$ is conjugate to $U_4(\rho,0,\omega^{-1})$.
 \end{proof}

\begin{lemma}\label{lemU7}
Assume $\char \F\neq 2$.
Let $R=\RR(N_2,D)\in \D$ be a nonabelian subgroup such that  $\d(\Z(R))=2$ and $\r(\Z(R))=1$. Then $R$ is 
conjugate to one of the following 
subgroups: 
$\RR_{(E_{1,2}-E_{2,1})}$, $U_3(1,-1,0)$ or $U_4(\rho,-1,0)$ for some $\rho\in 
\F^\square$.
\end{lemma}

\begin{proof}
Let $R=\RR(N_2,D)$ for some $D\in \Mat_3(\F)$.
By \eqref{dt} we must have $D=\begin{pmatrix} \bar\alpha_1 & \bar\alpha_2 & 0\\
\bar\beta_1 & \bar\beta_2 & 0 \\ 0 & 0 & 0 \end{pmatrix}$. 
Then, $R$ is conjugate via $g=I_5-\bar\beta_1 E_{4,5}$ to 
$$U_7(\alpha_1,\alpha_2,\beta_2)=
\begin{pmatrix}
1 & x_1 & x_2 & x_3 & x_4\\
0 & 1 & 0 & -x_2 & \alpha_1 x_1+ \alpha_2 x_2\\
0 & 0 & 1 & x_1 & \beta_2 x_2\\
0 & 0 & 0 & 1 & 0\\
0 & 0 & 0 & 0 & 1
    \end{pmatrix}
$$
for some $\alpha_1,\alpha_2,\beta_2\in \F$.
It is easy to see that $U_7(\alpha_1,\alpha_2,\beta_2)$ is  conjugate to some $\RR_A$ if and only if $\alpha_1=\alpha_2=\beta_2=0$. In this case $\d(R)=\r(R)=\k(R)=2$ and  $U_7(0,0,0)^g=\RR_{(E_{1,2}-E_{2,1})}$ where $\diag(I_3,(E_{2,1}-E_{1,2}))$.

Suppose now $(\alpha_1,\alpha_2,\beta_2)\neq (0,0,0)$ and let $\Delta=4\alpha_1\beta_2-\alpha_2^2$.
Notice that in this case $\d(R)=\r(R)=3$ and $\k(R)=2$. 
Assume $\Delta=0$ and consider
the function 
$\Psi: \A_{U_3(1,-1,0)} \rightarrow \A_R$  given by
$$\left\{\begin{array}{lll}
\Psi(t_1)=X_1, &  \Psi(t_2)=\alpha_2 X_1-2\alpha_1 X_2
& \textrm{ if } \alpha_1\neq 0, \\
\Psi(t_1)=X_2, & \Psi(t_2)=X_1 & \textrm{ if } \alpha_1=\alpha_2=0.
\end{array}\right.$$
Since $\Psi$  is an isomorphism, we obtain that
$U_7(\alpha_1,\alpha_2,\beta_2)$ is conjugate to $U_3(1,-1,0)$.
Now, assume $\Delta\neq 0$ and write $\Delta=\omega^2\rho$ for some $\omega \in \F^\ast$ and a unique $\rho \in 
\F^\square$.
The function 
$\Psi: \A_{U_4(\rho,-1,0)} \rightarrow \A_R$  given by
$$\left\{\begin{array}{lll}
\Psi(t_1)=\omega X_1, &  \Psi(t_2)=\alpha_2 X_1-2\alpha_1 X_2
& \textrm{ if } \alpha_1\neq 0, \\
\Psi(t_1)=\omega(\beta_2+1)X_1 - \alpha_2\omega X_2, & \Psi(t_2)=\alpha_2(\beta_2-1)X_1 - \alpha_2^2 X_2& \textrm{ if } \alpha_1=0,
\end{array}\right.$$
is an isomorphism: $R$ is conjugate to $U_4(\rho,-1,0)$.
\end{proof}

\begin{lemma}\label{lemU8}
Let $R=\RR(N_{3,\lambda},D)\in \D$, $\lambda \in \F^\ast$, be a nonabelian subgroup such that  $\d(\Z(R))=2$ 
and 
$\r(\Z(R))=1$. Then, $R$ is conjugate 
to one of the following subgroups: $\RR_{A}$, where $A= E_{1,1}+E_{2,1}+\lambda E_{2,2}$,  $U_3(0,1,1)$, 
$U_3(1,\mu,0)$  for some $\mu \in \F\setminus\{0,1\}$,
$U_4(\rho,\beta_1,\beta_2)$ for some $\rho\in \F^\square$ and $\beta_1,\beta_2\in \F$, or 
$U_5(1,1,\omega)$  where $\omega\in \F$ is such that $\omega+\lambda \in \G(\F)$.
\end{lemma}

\begin{proof}
Let $R=\RR(N_{3,\lambda},D)$ for some $\lambda \in \F^\ast$ and some $D\in \Mat_3(\F)$.
By \eqref{dt} $D=\begin{pmatrix} \bar \alpha_1 & \bar \alpha_2 & 0\\
\bar \beta_1 & \bar \beta_2 & 0 \\ 0 & 0 & 0 \end{pmatrix}$. 
Then, $R$ is conjugate via $g=I_5-\frac{\bar\beta_2}{\lambda} E_{4,5}$ to 
$$\tilde U_8(\lambda,\alpha_1,\alpha_2,\beta_1)=
\begin{pmatrix}
1 & x_1 & x_2 & x_3 & x_4\\
0 & 1 & 0 & x_1+x_2 & \alpha_1 x_1+\alpha_2 x_2\\
0 & 0 & 1 & \lambda x_2 & \beta_1 x_1\\
0 & 0 & 0 & 1 & 0\\
0 & 0 & 0 & 0 & 1
    \end{pmatrix}
$$
for some $\alpha_1,\alpha_2,\beta_1\in \F$.
It is easy to see that $\tilde U_8(\lambda,\alpha_1,\alpha_2,\beta_1)$ is  conjugate to some $\RR_A$ if and only if 
$\alpha_1=\alpha_2=\beta_1=0$. In this case $\d(R)=3$ and $\r(R)=\k(R)=2$ and  $\tilde U_8(\lambda,0,0,0)^g=\RR_A$, where 
$g=E_{1,1}+E_{2,3}+E_{5,4}-\lambda E_{2,2}+\lambda E_{3,3}+\lambda^2 E_{4,5}$ and
$A=E_{1,1}+E_{2,1}+\lambda E_{2,2}$.

So, suppose $(\alpha_1,\alpha_2,\beta_1)\neq (0,0,0)$: we have 
$\d(R)=\r(R)=3$ and $\k(R)=2$.
We may also assume $\beta_1=1$. Namely, if $\beta_1\neq 0$ it suffices to conjugate $\tilde U_8(\lambda,\alpha_1,\alpha_2,\beta_1)$ by 
$g=\diag(1,\beta_1,\beta_1,\beta_1^2,\beta_1)$.
If $\beta_1=0$ and $\alpha_2\neq 0$, then we conjugate by the matrix
$\diag\left(1,\begin{pmatrix} 1 & -1 \\ \lambda & 0\end{pmatrix},
\begin{pmatrix} \lambda & (\lambda\alpha_1-\alpha_2)\alpha_2^{-1} \\ 0 & 
-\lambda\alpha_2^{-1}\end{pmatrix}\right)$ and if
$\alpha_2=\beta_1=0$,  we conjugate by the matrix
$\diag\left(1,\begin{pmatrix} 0 & \alpha_1 \\ -\lambda\alpha_1 & \alpha_1\end{pmatrix},\begin{pmatrix} \lambda\alpha_1^2 & -\lambda\alpha_1^2 \\ 0 & \lambda\alpha_1\end{pmatrix}\right)$.

Now, set $U_8(\lambda, \alpha_1,\alpha_2)=\tilde U_8(\lambda,\alpha_1,\alpha_2,1)$
and
$\Delta=-\lambda(\alpha_2-1)(\alpha_1-\alpha_2+1)-1$.
Suppose $\alpha_2=1$. In this case $\Delta=-1$ that we write as $-1=\omega^2\rho$ for some $\omega \in \F^\ast$ and a unique 
$\rho \in \F^\square$.
If $\alpha_1\neq 0,1$, then 
$U_8(\lambda, \alpha_1,1)^g=
U_4\left(\rho,\frac{1}{1-\alpha_1},\frac{\alpha_1^2\lambda-\alpha_1+2}{\omega(\alpha_1-1)}\right)$,
where  
$$g=\diag\left(1,\begin{pmatrix} \alpha_1 & 0 \\ 1 & \omega \end{pmatrix},\begin{pmatrix} 0 & \frac{\omega \alpha_1^2}{\alpha_1-1} \\ \alpha_1 & 
\frac{\omega \alpha_1}{1-\alpha_1} \end{pmatrix}\right).$$
If $\alpha_1=0$ and $\char \F\neq 2$, then 
$U_8(\lambda, 0,1)^g=
U_4\left(\rho,1-\frac{2}{\lambda},-\frac{2(\lambda+1)}{\omega\lambda}\right)$,
where  
$g=\diag\left(1,\begin{pmatrix}
\omega^{-1} & 1 \\ \omega^{-1} & -1 \end{pmatrix},
\begin{pmatrix} 
0 &  -\frac{4}{\omega\lambda}\\
2\omega^{-2} & \frac{2(\lambda+2)}{\omega\lambda} \end{pmatrix}\right)$.
If $\alpha_1=0$ and $\char \F=2$, then
$U_8(\lambda, 0,1)^g=
U_5(1,1,\lambda+\mu+\mu^2)$,
where $g=\diag\left(1,\begin{pmatrix}
1 & 0 \\ \mu & 1 \end{pmatrix},\begin{pmatrix} 
0 &  1\\ 1 & \mu\end{pmatrix}\right)$.
If $\alpha_1=1$, then 
$U_8(\lambda, 1,1)^g=
U_4\left(\rho,0,-(\lambda+1)\omega^{-1}\right)$,
where 
$$g=\diag\left(1,\begin{pmatrix}
0 & 1 \\ \omega^{-1} & 1 \end{pmatrix},
\begin{pmatrix} 
0 &  -\omega^{-1}\\ -\omega^{-2} & 
-\lambda \omega^{-1} \end{pmatrix}\right).$$

Hence, we may now suppose $\alpha_2\neq 1$.  If $\Delta\neq 0$, we write $\Delta=\rho\omega^2$ for some 
$\omega\in\F^\ast$ and a unique $\rho \in 
\F^\square$.
Then 
$U_8(\lambda, \alpha_1,\alpha_2)^g=U_4(\rho,\alpha_2,-(\alpha_1\alpha_2\lambda-\alpha_1\lambda+\alpha_2+1)\omega^{-1})$
where 
$$g=\diag\left(1,\begin{pmatrix} 1 & \omega \\ \lambda(1-\alpha_2) & 0 \end{pmatrix}, \begin{pmatrix} \lambda (1-\alpha_2)^2 & 0 \\ \lambda(1-\alpha_2) & \omega\lambda(1-\alpha_2) \end{pmatrix}\right).$$
If $\Delta=0$, i.e. $\alpha_1=\frac{ \lambda(\alpha_2-1)^2-1}{\lambda (\alpha_2-1)}$, we need to distinguish  two cases.
If $\lambda(\alpha_2-1)^2 + \alpha_2\neq 0$, then $U_8(\lambda,\alpha_1,\alpha_2)^g=U_3(0,1,1)$, where
$$g=\diag\left(1, \begin{pmatrix} 1 &  0 &  \lambda(1-\alpha_2)-1 & 0 \\
 \lambda(1-\alpha_2) &  0 &  -\lambda&  0 \\
  0 &  \lambda(\alpha_2-1)^2&  0&  \lambda\alpha_2\\
 0 &  \lambda(1-\alpha_2) &  0  & \lambda^2(\alpha_2-1)
\end{pmatrix} \right),$$
and  if $\lambda=\frac{-\alpha_2}{(\alpha_2-1)^2}$, then $U_8(\lambda,\alpha_1,\alpha_2)^g=U_3(1,\alpha_2,0)$,
where
$$g=\diag\left(1,\begin{pmatrix}
  1 &  0 &  0 &  0 \\
 \frac{\alpha_2}{\alpha_2-1} &  0 &  1 &  0 \\
 0 &  -\alpha_2 &  0 & 1-\alpha_2^2\\
 0 & \frac{\alpha_2}{\alpha_2-1} &  0 & \alpha_2      
                \end{pmatrix}\right).$$
\end{proof}

\begin{cor}\label{cortab}
Let $R\in \D$ be a nonabelian  subgroup such that $\d(\Z(R))=2$ and $\r(\Z(R))=1$. Suppose that  
$R$ is not conjugate to any $\RR_A$, $A\in \Mat_3(\F)$. Then $R$ is conjugate to exactly one of the subgroups listed in Table \ref{tab3}.
\end{cor}

\begin{table}[t]
 $$\begin{array}{ccccl}\\
R & \d(R) & \r(R) & \k(R) & \mathrm{Conditions}\\\hline
U_3(0,1,0) &  3 & 3 & 3 \\
U_3(0,1,1) & 3 &  3 & 2\\
U_3(1,\lambda,0) & 3 & 3& 2 & \lambda \in \F\setminus\{0,1\}\\
U_4(\rho,\beta_1,\beta_2)  & 3 & 3 & 2 & \rho \in \F^\square,\; (\beta_1,\beta_2)\in \CC_\rho(\F) \\
U_5(1,1,\epsilon)  &  3 & 3 & 2 & \char \F=2, \; \epsilon\in \H(\F)\\
U_3(1,0,0) &  3 & 2 & 2
\end{array}$$
\caption{Representatives for conjugacy classes of nonabelian subgroups having $\d(\Z(R))=2$ and $\r(\Z(R))=1$.}\label{tab3}
\end{table}

\begin{proof}
In view of Lemmas \ref{lemU3}, \ref{U4}, \ref{lemU5}, \ref{lemU6}, \ref{lemU7} and \ref{lemU8}, $R$ is conjugate to one of the subgroups listed in Table \ref{tab3}. Since $\d(\RR_A)\leq 2$, it follows from Lemma \ref{lemU3} that  these subgroups are not conjugate to $\RR_A$ for any $A\in \Mat_3(\F)$.
Direct computations show that no pair of subgroups $U_3(0,1,1)$, $U_3(1,\lambda,0)$, $U_4(\rho,\beta_1,\beta_2)$ and $U_5(1,1,\epsilon)$ is conjugate in $\AGL_4(\F)$.
\end{proof}

We can now give the classification of the nonabelian subgroups in $\D$.

\begin{theorem}\label{main_nonab}
Let $\F$ be a field.
The distinct conjugacy classes of nonabelian regular subgroups in $\D$ can be represented by
the subgroups described in the first column of Table \ref{tab2}.
\end{theorem}

\begin{proof}
Let $R\in \D$ be a nonabelian regular subgroup. By the considerations given in Section \ref{sec2} we have the following possibilities: $(\d(\Z(R)),\r(\Z(R)))=(3,2)$, $(2,2)$ or $(2,1)$.
If  $(\d(\Z(R)),\r(\Z(R)))=(2,2)$ then $R$ is conjugate either to $U_2(0,1,0,1,1)$ or to $U_2(0,1,0,0,1)$, by Lemma \ref{lemU2}. If  $(\d(\Z(R)),\r(\Z(R)))=(2,1)$, then $R$ is conjugate to either one of the subgroups of Table \ref{tab3} or to a subgroup $\RR_A$, by Corollary \ref{cortab}.
If $(\d(\Z(R)),\r(\Z(R)))=(3,2)$, then $R$ is conjugate to a subgroup $\RR_A$.
The statement of the theorem is proved considering a set $\Pi_A(\F)$ of representatives for 
projective congruent 
classes of asymmetric matrices in $\Mat_3(\F)$.
\end{proof}

We would like to give  explicit sets $\CC_\rho$, at least when $\F$ is algebraically closed, the real field $\R$ or a finite field.
Observe that, for every field $\F$ of characteristic $\neq 2$,   $(-1,0)\in \CC_\rho$ for all $\rho \in \F^\square$, by Remark 
\ref{rem}(a).

Fixing $\rho\in \F^\square$, $\beta_1\in \F\setminus\{1\}$ and $\beta_2\in \F$ such that $(\beta_1,\beta_2)\neq (-1,0)$,
define the following polynomials in $\F[t]$:
\begin{equation}\label{f12}
\begin{array}{rclcrcl}
f_1(t) & = &\beta_1 t^2+\beta_2 t-\rho, & \quad &
f_2(t) & = & \beta_2 t^2-2\rho(\beta_1+1)t-\rho\beta_2,\\
g_1(t) & = &  t^2+\beta_2 t -\rho\beta_1, & \quad & g_2(t) & = & t^2+\rho.
\end{array}
\end{equation}

\begin{lemma}\label{f1}
Suppose that $\beta_1\neq 1$ and $(\beta_1,\beta_2)\neq (-1,0)$. If either $\beta_1=0$ or $f_1(t)$ is reducible, then 
$U_4(\rho,\beta_1,\beta_2)$ is conjugate to $U_4(\rho,0,\lambda)$ for some $\lambda \in \F$.
\end{lemma}

\begin{proof}
Clearly, the statement holds if $\beta_1=0$. So, assume $\beta_1\neq 0$.
Since $f_1(t)$ is reducible, we can take $a\in \F$ such that $f_1(a)=0$ and from \eqref{b3b4} we obtain that 
$U_4(\rho,\beta_1,\beta_2)$ is conjugate to $U_4\left(\rho,0,\frac{f_2(a)}{g_1(a)}\right)$
provided that $g_1(a)g_2(a)\neq 0$.
Suppose that $f_1(a)=g_1(a)g_2(a)=0$. Observe that $f_1(a)=g_1(a)=0$ implies $g_2(a)=0$, since 
$f_1(t)=g_1(t)+(\beta_1-1)g_2(t)$.
So, we may suppose that any root of $f_1(t)$ is also a root of $g_2(t)$.
Clearly, this implies that $\rho=-\omega^2$ for some $\omega\in \F$.
Take $a=a_i=(-1)^i\omega$, $i=0,1$.
If $f_1(a_i)=0$ for some $i$, then $\beta_2=(-1)^{i+1}\omega (\beta_1+1)$.
In this case, $f_1(t)=(t-a_i)(\beta_1 t-(-1)^i\omega)$ implies that
$\beta_1  (-1)^j\omega-(-1)^i\omega=0$ for some $j=0,1$, whence 
$\beta_1=-1$ and $(i,j)\in \{(0,1),(1,0)\}$. It follows that $\beta_2=0$, contradicting our assumption 
$(\beta_1,\beta_2)\neq (-1,0)$.
\end{proof}

In the following, given a subset $S$ of $\F^\ast$, we denote by $-S$ and $S^{-1}$ the sets $\{-s: s \in S\}$ and $\{s^{-1}: s \in S\}$, respectively.

\begin{cor}\label{corac}
Let $\F$ be a field with no quadratic extensions. 
\begin{itemize}
 \item[(a)] If $\char \F\neq 2$, then $\CC_1=\{(-1,0),(0,0)\}\cup \{(0,\lambda): \lambda \in \N(\F)\}$, where 
$\N(\F)$ is a fixed subset of $\F^\ast$ such that  $\F^\ast=\N(\F)\sqcup -\N(\F)$.
 \item[(b)] If $\char \F=2$, then $\CC_1=\{(0,\lambda): \lambda \in \F\}$.
\end{itemize}
\end{cor}

\begin{proof}
It follows from Lemma \ref{f1} and Remark \ref{rem}.
\end{proof}

\begin{proposition}\label{non_ac}
Let $\F=\overline\F$ be an algebraically closed field.
The distinct conjugacy classes of nonabelian regular subgroups in $\D$ can be represented by 
$$U_2(0,1,0,1,1),\quad U_2(0,1,0,0,1),\quad U_3(0,1,0),\quad U_3(0,1,1),$$ 
$$U_3(1,\lambda,0)\;(\lambda\in \F\setminus\{1\}),\quad \RR_D\;(D\in \Pi_A(\F))$$
and
$$\left\{\begin{array}{ll}
U_4(1,-1,0),\quad U_4(1,0,0),\quad U_4(1,0,\beta)\; (\beta \in \N(\F)) & \textrm{ if } \char \F\neq 2,\\
U_4(1,0,\beta)\;(\beta\in \F),\quad U_5(1,1,0)& \textrm{ if } \char \F=2,
\end{array}\right.$$
where, if $\char \F\neq 2$,
\begin{eqnarray*}
\Pi_A(\F) &= & \{E_{2,3},\; E_{1,2}+E_{2,3},\; E_{1,3}+E_{2,2}+E_{3,1}+E_{3,2},\; 
E_{3,2}+E_{3,3}-E_{2,3},\\
&&E_{1,1}+E_{2,3},\;  E_{1,1}+E_{3,2}+E_{3,3}-E_{2,3},\;E_{1,1}+E_{2,3}-E_{3,2},\\
&&E_{2,3}-E_{3,2},\; E_{2,3}+\mu E_{3,2},\; E_{1,1}+ E_{2,3}+\mu E_{3,2}: \mu \in \M(\F) \}, 
\end{eqnarray*}
and $\M(\F)$ and $\N(\F)$ are subsets of $\F^\ast$ such that
$\F\setminus \{0,\pm 1 \}=\M(\F)\sqcup \M(\F)^{-1}$
and $\F^\ast=\N(\F)\sqcup -\N(\F)$;
if $\char \F=2$,
\begin{eqnarray*}
\Pi_A(\F) &= & \{E_{2,3},\; E_{1,1}+E_{2,3},\; E_{1,2}+E_{2,3},\;E_{1,3}+E_{2,2}+E_{3,1}+E_{3,2},\\
&&E_{2,3}+\mu E_{3,2}:\mu \in \M(\F)\}
\end{eqnarray*}
and $\M(\F)$ is a  subset of $\F^\ast$ such that
$\F\setminus \{0,1 \}=\M(\F)\sqcup \M(\F)^{-1}$.
\end{proposition}

\begin{proof}
By Theorem \ref{main_nonab} and Corollary \ref{corac} we are left to determine the set $\Pi_A(\F)$.
Applying  \cite[Theorem 2.1]{HS}, the elements of $\Pi_A(\F)$ are the asymmetric matrices obtained as diagonal 
sum of the following matrices:
$$
\begin{pmatrix}  0\end{pmatrix}, 
\begin{pmatrix} 1 \end{pmatrix},
\begin{pmatrix}  0 & 1 \\ 0 & 0\end{pmatrix}, 
\begin{pmatrix}  0 & 1 \\ \mu & 0\end{pmatrix}, 
\begin{pmatrix}  0 & -1 \\ 1 & 1\end{pmatrix}, 
\begin{pmatrix}  0 & 1 & 0 \\ 0 & 0 &1 \\ 0 & 0 & 0\end{pmatrix}, 
\begin{pmatrix}  0 & 0 & 1 \\ 0 & 1 &0 \\ 1& 1 & 0\end{pmatrix},
$$
where $\mu\in\M(\F)$. If $\char \F\neq 2$ we also allow $\mu=-1$; if $\char \F=2$ we must omit the matrices
$$\diag\left(1 ,\begin{pmatrix}  0 & 1 \\ \mu & 0\end{pmatrix} \right).$$
\end{proof}

\begin{lemma}\label{q3}
Let $\rho=1$ and $\char \F\neq 2$.
Suppose that there exists an element $\iota \in\F$ such that $\iota^2=-1$. Then
$$\CC_1=\{(-1,0),(0,0)\}\cup \{(\lambda+1,2\iota): \lambda \in \F^\ast\}.$$
\end{lemma}

\begin{proof}
We show that every subgroup $U_4(1,\beta_1,\beta_2)$, $\beta_1\neq 1$, is conjugate to a unique subgroup $U_4(1,\beta_3,\beta_4)$ with $(\beta_3,\beta_4)\in \CC_1$ as in the statement. This clearly holds if $\beta_4 = \pm 2\iota$, so suppose $\beta_4\neq \pm 2\iota$.
Keeping the notation of \eqref{f12}, let $h_1(t)=f_2(t)-2\iota g_1(t)\in \F[t]$.
Then, $a=\frac{2\beta_1+\iota \beta_2}{\beta_2-2\iota}$ is a root of $h_1(t)$. It follows that $R$ is conjugate to $U_4(1,\lambda+1,2\iota)$ for some $\lambda \in \F^\ast$, provided that $g_1(a)g_2(a)\neq 0$.
Notice that $g_1(a)g_2(a)=0$ if and only if $(\beta_1+\iota \beta_2+1)(\beta_2^2+4\beta_1)=0$.

Suppose that $\beta_2^2+4\beta_1=0$. If  $\beta_2\neq 0$, take $b=2\beta_2^{-1}$. In this case,
$f_1(b)=f_2(b)=0$, $g_1(b)=\frac{(\beta_2^2+4)^2}{4\beta_2^2}\neq 0$ and $g_2(b)=\frac{\beta_2^2+4}{\beta_2^2}\neq 0$.
It follows that $R$ is conjugate to $U_4(1,0,0)$ (this clearly holds even if $\beta_2=0$).
Now, suppose that $\beta_1+\iota \beta_2+1=0$ and let $h_2(t)=f_2(t)+2\iota g_1(t)$. 
Take $c=-\frac{3\iota \beta_2+2}{\beta_2+2\iota}$.
Then $h_2(c)=0$, $g_1(c)=\frac{-2\iota \beta_2(\beta_2-2\iota)^2)}{(\beta_2+2\iota)^2}\neq 0$ and $g_2(c)=\frac{-8\beta_2(\beta_2-2\iota)}{(\beta_2+2\iota)^2}\neq 0$.
It follows that $R$ is conjugate to $U_4(1,-3,-2\iota)$ and then to $U_4(1,-3,2\iota)$, by Remark \ref{rem}.

Finally, it is an easy computation to verify that $U_4(1,\lambda_1+1,2\iota)$ is conjugate to $U_4(1,\lambda_2+1,2\iota)$ if and only if $\lambda_1=\lambda_2$.
\end{proof}

Clearly the previous lemma gives an alternative set $\CC_1$ for algebraically closed field of characteristic $\neq 2$.
In a very similar way we can prove the following.

\begin{lemma}\label{q1}
Let $\rho=-1$, $\char \F\neq 2$ and suppose that $-1\not \in (\F)^2$. Then
$$\CC_{-1}=\{(-1,0),(0,0)\}\cup \{(\lambda+1,2): \lambda \in \F^\ast\}.$$
\end{lemma}

\begin{lemma}\label{f2}
Suppose that $\char \F\neq 2$, $\beta_1\neq 0,1$, $\beta_2\neq 0$ and $(\beta_1,\beta_2)\neq (-1,0)$.
If $f_1(t)$ is irreducible and $f_2(t)$ is reducible, then $U_4(\rho,\beta_1,\beta_2)$ is conjugate to 
$U_4(\rho,\lambda,0)$ for some $\lambda \in \F\setminus\{0,1\}$.
\end{lemma}

\begin{proof}
Since $f_2(t)$ is reducible, we can take $a\in \F$ such that $f_2(a)=0$ and from \eqref{b3b4} we obtain that 
$U_4(\rho,\beta_1,\beta_2)$ is conjugate to $U_4\left(\rho,\frac{f_1(a)}{g_1(a)},0\right)$
provided that $g_1(a)g_2(a)\neq 0$.
So, first, suppose that  $f_2(a)=g_2(a)=0$ for some $a\in \F$. As in  Lemma \ref{f1}, this implies that 
$\rho=-\omega^2$ for some $\omega\in \F^\ast$ and  $a=a_i=(-1)^i\omega$ for some $i=0,1$.
In particular, we get $\beta_2=(-1)^{i+1}\omega(\beta_1+1)$ (whence $\beta_1\neq -1$).
In this case, we obtain that $f_1(a_i)=0$, an absurd, since $f_1(t)$ is irreducible.
Next, suppose that any root of $f_2(t)$ is a root of $g_1(t)=0$ (but not a root of $g_2(t)$).
Let $r$ be the resultant between $f_2(t)$ and $g_1(t)$ with respect to $t$. From $r=0$ we get
$(\beta_2^2+4\rho\beta_1)(\rho(\beta_1+1)^2 + \beta_2^2)=0$.
Since $f_1(t)$ is irreducible, this implies $\rho(\beta_1+1)^2 +\beta_2^2=0$ and so
$\rho=-\omega^2$ for some $\omega \in \F^\ast$ and $\beta_2=\pm \omega(\beta_1+1)$.
As before, this condition implies that $f_1(t)$ is reducible, a contradiction. 
\end{proof}

\begin{cor}\label{R1}
Let $\F=\R$. Then 
$$\CC_1=\{(0,\lambda): \lambda \in [0,+\infty)\}\cup \{(\lambda,0): \lambda \in [-1,0)\}.$$
\end{cor}

\begin{proof}
The result follows from Lemmas \ref{f1} and \ref{f2} and Remark \ref{rem}, observing that $f_2(t)\in \R[t]$ is reducible provided that $\beta_2\neq 0$. 
\end{proof}

From Theorem \ref{main_nonab} and \cite[Theorem 2.1]{HS} we obtain the following classification.

\begin{proposition}\label{nonab_real}
The distinct conjugacy classes of nonabelian regular subgroups in $\Delta_4(\R)$ can be represented by:
$$U_2(0,1,0,1,1),\quad U_2(0,1,0,0,1),\quad U_3(0,1,1),\quad U_3(0,1,0),\quad \RR_D\;(D\in \Pi_A(\R)),$$
$$U_3(1,\lambda+1,0)\; (\lambda\in \R^\ast),\quad 
U_4(1,0,\lambda)\;(\lambda \in [0,+\infty)),\quad U_4(1,\lambda,0)\;(\lambda \in [-1,0)),$$
$$U_4(-1,0,0),\;\;U_4(-1,-1,0),\;\;U_4(-1,\lambda+1,2)\;(\lambda \in \R^\ast),$$
where  
\begin{eqnarray*}
\Pi_A(\R) & = & \{E_{1,2}+E_{2,3},\;E_{1,3}+E_{2,2}+E_{3,1}+E_{3,2},\;E_{1,2}\pm E_{3,3},\; E_{1,2}-E_{2,1},\\
&& E_{1,2}-E_{2,1}\pm E_{3,3},\; E_{1,2}+\mu E_{2,1},\;E_{1,2}-E_{2,1}-E_{2,2}\pm E_{3,3},\;E_{1,2},\\
&& E_{1,2}+\mu E_{2,1}\pm E_{3,3},\;\nu E_{1,1}-E_{1,2}+E_{2,1}+\nu E_{2,2},\; E_{1,2}-E_{2,1}-E_{2,2},\\
&&\nu E_{1,1}-E_{1,2}+E_{2,1}+\nu E_{2,2}\pm E_{3,3}: \mu \in (0,1), \nu \in \R^\ast\}.
\end{eqnarray*}
\end{proposition}

\begin{lemma}\label{ii}
Suppose that $\char \F\neq 2$, $\beta_1\neq 0,1$ and $(\beta_1,\beta_2)\neq(-1,0)$. Let $\F^\square=\{1,\xi\}$ and let $\vartheta\in \F$ such that 
$\vartheta^2+1=\sigma^2 \xi$ for some $\sigma\in \F^\ast$.
If $f_1(t)$ and $f_2(t)$ are both irreducible, then $U_4(1,\beta_1,\beta_2)$ and 
$U_4(\xi,\beta_1,\beta_2)$ are conjugate, respective, to $U_4(1,\lambda_1,\vartheta(\lambda_1+1))$ and 
$U_4(\xi,-1,\lambda_2)$ for some $\lambda_1\in \F\setminus\{\pm 1\}$ and $\lambda_2 \in \F^\ast$.
\end{lemma}

\begin{proof}
Since $f_1(t)$ and $f_2(t)$ are both irreducible polynomials, then $\beta_2^2+4\rho\beta_1=\omega^2 \xi$ and
$(\beta_1+1)^2\rho^2+\rho\beta_2^2=\nu^2\xi$ for some $\omega,\nu\in \F^\ast$.
First, consider $\rho=1$. The statement is obvious if $\beta_2=\vartheta(\beta_1+1)$. 
So, assume $\beta_2\neq \vartheta(\beta_1+1)$ and let $h(t)=f_2(t)-\vartheta(f_1(t)+g_1(t))$.
Since the discriminant of $h(t)$ is $4(\vartheta^2+1)\nu^2\xi=(2\sigma\nu\xi)^2$, there exists
$a\in \F$ such that $h(a)=0$. 
Now, denote by $r_i$ the resultant between $h(t)$ and $g_i(t)$ with respect to $t$: we obtain
$r_1=\nu^2\xi((\beta_1-1)^2\vartheta^2-\omega^2\xi)$ and $r_2=4(\vartheta^2+1)\nu^2\xi$.
From $r_1=0$ we get $\xi=((\beta_1-1)\vartheta\omega^{-1})^2$, an absurd;
from $r_2=0$ we obtain  $\vartheta^2+1=0$, which is in contradiction with the initial assumption on $\vartheta$.
Hence, $g_1(a)g_2(a)\neq 0$ and so $U_4(1,\beta_1,\beta_2)$ is conjugate to $U_4\left(1,\frac{f_1(a)}{g_1(a)}, 
\frac{\vartheta(f_1(a)+g_1(a))}{g_1(a)}\right)$.

Next, consider $\rho=\xi$.
The statement is obvious if $\beta_1=-1$. So, assume $\beta_1\neq -1$ and let $h(t)=f_1(t)+g_1(t)$. Since the 
discriminant of $h(t)$ is $4\nu^2$, there exists $a\in \F$ such that $h(a)=0$.
Clearly, if $h(a)=g_1(a)=0$, then $f_1(a)=0$, contradicting the assumption that $f_1(t)$ is irreducible.
So, suppose that $h(a)=g_2(a)=0$ and denote by $r$ the resultant between $h(t)$ and $g_2(t)$ with respect to $t$.
Since $r=4\nu^2\xi\neq 0$, we obtain that $g_2(a)\neq 0$ 
and so
$U_4(\xi,\beta_1,\beta_2)$ is conjugate to $U_4\left(\xi,-1,\frac{f_2(a)}{g_1(a)}\right)$.
\end{proof}

\begin{lemma}\label{repr}
Let $\F=\F_q$ be a finite field and keep the previous notation.
\begin{itemize}
\item[(a)] Suppose $q$ even. Then we can take $\F^\square=\{1\}$ and
$$\CC_1=\{(0,\lambda): \lambda \in \F\}\cup\left\{\left(1+\kappa_a^{-1},a \kappa_a^{-1}\right): a \in \F\setminus\{0,1\}\right\},$$
where, for every $a\in \F\setminus\{0,1\}$,  $\kappa_a\in \F^\ast$ is a fixed element 
such that $\kappa_a^2+\kappa_a\not\in a^2\G(\F)$.
In particular, $|\CC_1|=2q-2$.
\item[(b)] Suppose $q\equiv 1\pmod 4$. Then we can take $\F^\square=\{1,\xi\}$ and
\begin{eqnarray*}
\CC_1 & = & \{(-1,0),(0,0)\}\cup \{(\lambda+1,2\iota): \lambda \in\F_q^\ast\},\\
\CC_\xi & = & \{(-1,0),(0,0)\}\cup \{(0,\lambda): \lambda \in \N(\F_q)\}\cup \{(\lambda,0): \lambda \in \P(\F_q) \}\cup \\
&& \cup \{(-1, \lambda): \lambda \in \QQ(\F_q) \},
\end{eqnarray*}
where $\iota^2=-1$ and 
$$\F_q^\ast\setminus\left(\{\pm 1\}\cup \{\omega^2\xi: \omega\in \F_q\}\right)=\P(\F_q)\sqcup -\P(\F_q)$$
and 
$$\F_q^\ast\setminus\left\{ a+a^{-1}\xi: a \in \F_q^\ast\right\}=\QQ(\F_q)\sqcup -\QQ(\F_q).$$
In particular, $|\CC_1|+|\CC_\xi|=2q+1$.
\item[(c)] Suppose $q\equiv 3\pmod 4$. Then we can take  $\F^\square=\{1,-1\}$ and
\begin{eqnarray*}
\CC_1 & = & \{(-1,0),(0,0)\}\cup \{(0,\lambda): \lambda \in \N(\F_q)\}\cup \{(\lambda,0): \lambda \in \P(\F_q) \}\cup \\
&& \cup \{(\lambda, \vartheta(\lambda+1)): \lambda \in \S(\F_q) \},\\
\CC_{-1} & = & \{(-1,0),(0,0)\}\cup \{(\lambda+1,2): \lambda \in \F_q^\ast\},
\end{eqnarray*}
where $\vartheta$ is such that $\vartheta^2+\sigma^2+1=0$ for some $\sigma\in \F^\ast$ and
$$\F_q^\ast\setminus\left(\{\pm 1\}\cup \{\omega^2: \omega\in \F_q^\ast\}\right)=\P(\F_q)\sqcup -\P(\F_q)$$
and 
$$\F_q^\ast\setminus\left(\{\pm 1\}\cup \left\{\frac{1-a\vartheta}{a(a+\vartheta)}: a \in \F_q^\ast\setminus\{-\vartheta\}\right\}\right)=\S(\F_q)\sqcup \S(\F_q)^{-1}.$$
In particular, $|\CC_1|+|\CC_{-1}|=2q+1$.
\end{itemize}
\end{lemma}
 
\begin{proof}
Let $f_1(t),f_2(t),g_1(t),g_2(t)$ as in \eqref{f12}.
Suppose $q$ even. If either $\beta_1\neq 0$ or $f_1(t)$ is reducible, then $U_4(1,\beta_1,\beta_2)$ is conjugate to $U_4(1,0,\lambda)$ for a unique $\lambda \in \F$, by Lemma \ref{f1} and Remark \ref{rem}.
Now, suppose $\beta_1\neq 0,1$ and $f_1(t)$ irreducible.
Then $\beta_2\neq 0$, $\beta_1\beta_2^{-2}\not \in \G(\F)$ and $g_1(t)$ is also irreducible. We obtain that $U_4(1,\beta_1,\beta_2)$ is conjugate to $U_4(1,\beta_3,\beta_4)$ ($\beta_3\neq 0,1$, $\beta_4\neq 0$, $\beta_3\beta_4^{-2}\not \in \G(\F)$) if and only if $\beta_2(\beta_3+1)=\beta_4(\beta_1+1)$.
So, $U_4(1,\beta_1,\beta_2)$ is conjugate to  $U_1(1,1+\kappa_a^{-1},a\kappa_a^{-1})$, where $a=\beta_2(\beta_1+1)^{-1}\neq 0,1$ and  $\kappa_a\in \F^\ast$ is chosen in such a way that $\kappa_a^2+\kappa_a\not \in a^2\G(\F)$.

Suppose $q\equiv 1 \pmod 4$. The set $\CC_1$ is given in Lemma \ref{q3}. So, consider the case $\rho=\xi$ with $(\beta_1,\beta_2)\neq (-1,0)$.
If either $\beta_1=0$ or $f_1(t)$ is reducible, then $R=U_4(\xi,\beta_1,\beta_2)$ is conjugate to $U_4(\xi,0,\lambda)$ for a unique $\lambda \in \N(\F_q)$, by Lemma \ref{f1} and Remark \ref{rem}. So, assume $\beta_1\neq 0$ and $f_1(t)$ irreducible.
If $f_2(t)$ is reducible, then $R$ is conjugate to $U_4(\xi,\lambda,0)$ for some $\lambda\neq 0,1$, by Lemma \ref{f2}.
Recall that, by Remark \ref{rem}, $U_4(\xi,\lambda_1,0)$ is conjugate to $U_4(\xi,\lambda_2,0)$ if and only if either 
$\lambda_2=\lambda_1$ or $\lambda_2=\lambda_1^{-1}$; $U_4(\xi,\lambda_1,0)$ is conjugate to $U_4(\xi,0,\lambda_2)$ if and only if $\lambda_1=\xi \omega^2$ and $\lambda_2=2\xi\omega(\xi\omega^2-1)^{-1}$ for some $\omega\in \F$; 
$U_4(\xi,-1,\lambda_1)$ is conjugate to $U_4(\xi,-1,\lambda_2)$ if and only if $\lambda_2=\pm \lambda_1$.
Finally, if $\beta_1\neq 0,1$ and both $f_1(t)$ and $f_2(t)$ are reducible, then $R$ is conjugate to $U_4(\xi, -1, \lambda)$ for some $\lambda\in \F$, by Lemma \ref{ii}. Now, $U_4(\xi,-1,\lambda_1)$ is not conjugate to $U_4(\xi,\lambda_2,0)$ if $\lambda_1\neq 0$. Furthermore, $U_4(\xi,-1,\lambda_1)$ is conjugate to $U_4(\xi,0,\lambda_2)$
if and only if $\lambda_1=a+\xi a^{-1}$ and $\lambda_2=\frac{\pm a^2-\xi}{2a}$ for some $a\in\F^\ast$.

Suppose $q\equiv 3 \pmod 4$. The set $\CC_{-1}$ is given in Lemma \ref{q1}. So, consider the case $\rho=1$ with $(\beta_1,\beta_2)\neq (-1,0)$. Proceeding in a way similar to what we did for the case $q\equiv 1 \pmod 4$, $R=U_4(1,\beta_1,\beta_2)$ is conjugate to a subgroup having one of the following shape $U_4(1,0,\lambda)$ for a unique $\lambda \in \N(\F_q)$,
$U_4(1,\lambda,0)$ for some $\lambda\neq 0,1$, or $U_4(1,\lambda, \vartheta(\lambda+1))$ for some $\lambda\neq 1$.
Now, $U_4(1,\lambda_1,0)$ is conjugate to $U_4(1,\lambda_2,0)$ if and only if either 
$\lambda_2=\lambda_1$ or $\lambda_2=\lambda_1^{-1}$; $U_4(1,\lambda_1,0)$ is conjugate to $U_4(1,0,\lambda_2)$ if and only if $\lambda_1=\omega^2$ and $\lambda_2=2\omega(\omega^2-1)^{-1}$ for some $\omega\in \F\setminus\{\pm 1\}$;
$U_4(1,\lambda_1,\vartheta(\lambda_1+1))$ is conjugate to $U_4(1,\lambda_2,\vartheta(\lambda_2+1))$ if and only if $\lambda_2=\lambda_1^{\pm 1}$.

Furthermore, $U_4(1,\lambda_1,\vartheta(\lambda_1+1))$ is conjugate to $U_4(1,0,\lambda_2)$ if and only if
$\lambda_1=\left(\frac{1-a\vartheta}{a(a+\vartheta)}\right)^{\pm 1}$, $\lambda_2=\mp \frac{a^2\vartheta-2a-\vartheta}{a^2+2a\vartheta-1}$,
$a\neq 0,-\vartheta,\vartheta^{-1}$; $U_4(1,\lambda,\vartheta(\lambda_1+1))$ is not conjugate to $U_4(1,\lambda_2,0)$.
\end{proof}

\begin{proposition}\label{non_f}
Let $\F=\F_q$ be a finite field. If $q$ is odd, there are exactly $5q+9$ 
distinct conjugacy classes of nonabelian regular subgroups in $\Delta_4(\F_q)$, that can be represented 
by:
$$U_2(0,1,0,1,1),\;\; U_2(0,1,0,0,1),\;\; U_3(0,1,1),\;\;U_3(0,1,0),\;\; U_3(1,\lambda,0)\; (1\neq\lambda\in \F_q),$$
$$\RR_D\;(D\in \Pi_A(\F_q)),\;\;U_4(1,\beta_1,\beta_2)\;((\beta_1,\beta_2)\in \CC_1),\;\;
U_4(\xi,\beta_1,\beta_2)\;((\beta_1,\beta_2)\in \CC_\xi),$$
where
\begin{eqnarray*}
\Pi_A(\F_q) & = & \{E_{2,3}-E_{3,2},\; E_{1,1}+E_{2,3}-E_{3,2},\; E_{2,2}+E_{2,3}+ \lambda E_{3,3},\\
&&E_{1,1}+E_{2,2}+E_{2,3}+\lambda E_{3,3},\; \xi E_{1,1}+E_{2,2}+2E_{2,3}+E_{3,3},\\
&&E_{1,3}+E_{2,3}+E_{3,2},\; E_{1,1}+E_{1,3}+E_{2,3}+E_{3,2}: \lambda \in \F_q\}
\end{eqnarray*}
and $x^2-\xi$ is a fixed irreducible polynomial of $\F_q[x]$ (take $\xi=-1$ if $q\equiv 3 \pmod 4$).

If $q$ is even,  there are exactly $5q+6$ 
distinct conjugacy classes of nonabelian regular subgroups in $\Delta_4(\F_q)$, that can be represented 
by:
$$U_2(0,1,0,1,1),\;\; U_2(0,1,0,0,1),\;\;  U_3(0,1,1),\;\;  U_3(0,1,0),\;\;
U_3(1,\lambda,0)\; (1\neq \lambda\in \F_q),$$ 
$$U_5(1,1,0),\;\; U_5(1,1,\xi), \;\; \RR_D\;(D\in \Pi_A(\F_q)),\;\;U_4(1,\beta_1,\beta_2)\;((\beta_1,\beta_2)\in \CC_1),$$
where
\begin{eqnarray*}
\Pi_A(\F_q) & =& \{E_{2,2}+E_{2,3}+\lambda E_{3,3},\;E_{1,1}+E_{2,2}+E_{2,3}+\lambda E_{3,3},\\
&&E_{1,3}+E_{2,3}+E_{3,2},\;E_{1,1}+E_{1,3}+E_{2,3}+E_{3,2},\\
&&\xi E_{1,1}+E_{1,3}+E_{2,3}+E_{3,2}+E_{3,3}:\lambda \in \F_q\},
\end{eqnarray*}
and $x^2+x+\xi$ is a fixed irreducible polynomial of $\F_q[x]$.
\end{proposition}

\begin{proof}
The set $\Pi_S(\F_q)$ have been determined in \cite[Theorem 4]{W1}, so the result follows from  Theorem \ref{main_nonab} and Lemma \ref{repr}. 
\end{proof}

We conclude recalling that Theorem \ref{main} follows immediately from Theorems \ref{main_ab} and \ref{main_nonab}, applying Proposition \ref{prop34}.

\end{document}